\documentclass[11pt]{amsart}

\usepackage{amscd}
\usepackage{amsmath, amssymb}
\usepackage{amsfonts}
\newcommand{\de}{\partial}
\newcommand{\db}{\overline{\partial}}

\newcommand{\ddt}{\frac{\partial}{\partial t}}

\newcommand{\ddbar}{\sqrt{-1} \partial \overline{\partial}}
\newcommand{\Ric}{\mathrm{Ric}}
\newcommand{\ov}[1]{\overline{#1}}
\newcommand{\mn}{\sqrt{-1}}

\newcommand{\tr}[2]{\mathrm{tr}_{#1}{#2}}
\newcommand{\ti}[1]{\tilde{#1}}
\newcommand{\vp}{\varphi}

\newcommand{\ve}{\varepsilon}
\newcommand{\of}{\omega_{\textrm{flat}}}
\newcommand{\Om}{\Omega}
\newcommand{\bd}{\begin{enumerate}}
\newcommand{\ed}{\end{enumerate}}
\newcommand{\btheorem}{\begin{theorem}}
\newcommand{\etheorem}{\end{theorem}}
\newcommand{\bproposition}{\begin{proposition}}
\newcommand{\eproposition}{\end{proposition}}
\newcommand{\bdefinition}{\begin{definition}}
\newcommand{\edefinition}{\end{definition}}
\newcommand{\bcorollary}{\begin{corollary}}
\newcommand{\ecorollary}{\end{corollary}}
\newcommand{\bproof}{\begin{proof}}
\newcommand{\eproof}{\end{proof}}
\newcommand{\bremark}{\begin{remark}}
\newcommand{\eremark}{\end{remark}}
\newcommand{\eexample}{\end{example}}
\newcommand{\bexample}{\begin{example}}
\newcommand{\elemma}{\end{lemma}}
\newcommand{\blemma}{\begin{lemma}}
\newcommand{\ee}{\end{eqnarray*}}
\newcommand{\be}{\begin{eqnarray*}}

\newcommand{\Tf}{T^{0}}
\newcommand{\Kod}{\textrm{Kod}}

\renewcommand{\leq}{\leqslant}
\renewcommand{\geq}{\geqslant}
\renewcommand{\le}{\leqslant}
\renewcommand{\ge}{\geqslant}
\newcommand{\gf}{g_{\textrm{flat}}}
\newcommand{\mS}{\mathcal{S}}

\numberwithin{equation}{section}

\begin{document}
\newtheorem{claim}{Claim}
\newtheorem{theorem}{Theorem}[section]
\newtheorem{lemma}[theorem]{Lemma}
\newtheorem{corollary}[theorem]{Corollary}
\newtheorem{proposition}[theorem]{Proposition}
\newtheorem{question}{question}[section]
\newtheorem{conjecture}[theorem]{Conjecture}

\theoremstyle{definition}
\newtheorem{remark}[theorem]{Remark}

\title[Collapsing of the Chern-Ricci flow on elliptic surfaces]{Collapsing of the Chern-Ricci flow on elliptic surfaces$^*$}

\author[V. Tosatti]{Valentino Tosatti}
\thanks{$^{*}$Research supported in part by NSF grants DMS-1105373 and  DMS-1236969.  The first-named author is supported in part by a Sloan Research Fellowship.}
\address{Department of Mathematics, Northwestern University, 2033 Sheridan Road, Evanston, IL 60208}
\author[B. Weinkove]{Ben Weinkove}
\author[X. Yang]{Xiaokui Yang}
\begin{abstract}  We investigate the Chern-Ricci flow, an evolution equation of Hermitian metrics generalizing the K\"ahler-Ricci flow, on elliptic bundles over a Riemann surface of genus greater than one.  We show that, starting at any Gauduchon metric, the flow collapses the elliptic fibers
and the metrics converge to the pullback of a K\"ahler-Einstein metric from the base.
Some of our estimates are new even for the K\"ahler-Ricci flow.  A consequence of our result is that, on every minimal non-K\"ahler surface of Kodaira dimension one, the Chern-Ricci flow converges in the sense of Gromov-Hausdorff to an orbifold K\"ahler-Einstein metric on a Riemann surface.
\end{abstract}

\maketitle

\section{Introduction}

The Chern-Ricci flow is an evolution equation for Hermitian metrics on complex manifolds.  Given a starting Hermitian metric $g_0$, which we represent as a real $(1,1)$ form $\omega_0 = \sqrt{-1} (g_0)_{i\ov{j}}dz^i \wedge d\ov{z}^j$, the Chern-Ricci flow is given by
\begin{equation} \label{crf0}
\frac{\de}{\de t}\omega=-\Ric(\omega), \qquad \omega|_{t=0} =\omega_0,
\end{equation}
where $\Ric(\omega) := - \ddbar \log \det g$ is the Chern-Ricci form of $\omega$.  In the case when $g_0$ is K\"ahler, namely $d\omega_0=0$, (\ref{crf0}) coincides with the K\"ahler-Ricci flow.

 The Chern-Ricci flow was first introduced by Gill \cite{G} in the setting of manifolds with $c_1^{\textrm{BC}}(M)=0$, where $c_1^{\textrm{BC}}(M)$ is the first Bott-Chern class given by$$c_1^{\textrm{BC}}(M) = [ \Ric(\omega)] \in H^{1,1}_{\textrm{BC}}(M, \mathbb{R}) =\frac{ \{ \textrm{closed real $(1,1)$-forms} \}}{  \ddbar( C^{\infty}(M))},$$
for any Hermitian metric $\omega$.  Making use of an estimate for the complex Monge-Amp\`ere equation \cite{Ch, TW3},
Gill showed that solutions to (\ref{crf0}) on manifolds with $c_1^{\textrm{BC}}(M)=0$ exist for all time and converge to Hermitian metrics with vanishing Chern-Ricci form.  Gill's theorem generalizes the convergence result of Cao \cite{Ca} for the K\"ahler-Ricci flow (which made use of estimates of Yau \cite{Y1}).

The first and second named authors investigated the Chern-Ricci flow on more general manifolds \cite{TW, TW2}  and proved a number of further results.  It was shown in particular that the maximal existence time for the flow can be determined from the initial metric; that the Chern-Ricci flow on manifolds with negative first Chern class smoothly deforms Hermitian metrics to K\"ahler-Einstein metrics; that when starting on a complex surface with Gauduchon initial metric $\omega_0$ (meaning $\de\db \omega_0=0$), the Chern-Ricci flow exists until the volume of the manifold or a curve of negative self-intersection goes to zero; and that on surfaces with nonnegative Kodaira dimension the Chern-Ricci flow contracts an exceptional curve when one exists.  There are analogues of all of these results for the K\"ahler-Ricci flow \cite{Ca, FIK, TZha, SW1}.

For the purpose of this discussion it will be useful to make reference to the following condition:
$$(*) \quad \textrm{$M$ is a minimal non-K\"ahler complex surface and $\omega_0$ is Gauduchon}.$$   Surfaces which satisfy $(*)$ are of significant interest as they are not yet completely classified.  Recall that a surface is minimal if it contains no $(-1)$-curves and every complex surface is birational to a minimal one by via a finite sequence of blow downs.  Every complex surface admits a Gauduchon metric, and the Gauduchon condition is preserved by the Chern-Ricci flow.

We remark that Streets-Tian \cite{StT} earlier proposed the use of a different parabolic flow, called the Pluriclosed Flow, to study complex non-K\"ahler surfaces (see  Section 2 of \cite{TW} for some discussion on how this flow differs from the Chern-Ricci flow).

The Kodaira-Enriques classification (see \cite{bhpv}) tells us that manifolds $M$ satisfying $(*)$ fall into one of the following groups:

\begin{itemize}
\item $\Kod(M)=1$.  Minimal non-K\"ahler properly elliptic surfaces.
\item  $\Kod(M)=0$.  Kodaira surfaces.
\item $\Kod(M)=-\infty$.  Class VII surfaces which have either:
\begin{itemize}
\item[$\diamond$] $b_2(M)=0$.  Hopf surfaces or Inoue surfaces by \cite{In, LYZ, T0}.
\item[$\diamond$] $b_2(M)=1$.  These are classified by \cite{Na, T1}.
\item[$\diamond$] $b_2(M)>1$.  Still unclassified.
\end{itemize}
\end{itemize}

Here $\Kod(M)$ is the Kodaira dimension of $M$.
The result of Gill \cite{G} shows that when $\Kod(M)=0$ the Chern-Ricci flow exists for all time and converges to a Chern-Ricci flat metric.

In \cite{TW2}, \emph{explicit examples} of solutions to the Chern-Ricci flow were found on
all $M$ with $\Kod(M)=1$, for all Inoue surfaces and for a large class of Hopf surfaces.  In particular, it was shown that for any $M$ with $\Kod(M)=1$ there exists an explicit solution $\omega(t)$ of the Chern-Ricci flow for $t\in [0,\infty)$ with the property that as $t\rightarrow \infty$ the normalized metrics $\omega(t)/t$ converge in the sense of Gromov-Hausdorff to $(C, d_{\textrm{KE}})$ where $C$ is a Riemann surface and $d_{\textrm{KE}}$ is the distance function induced by an orbifold K\"ahler-Einstein metric on $C$.

The main result of this paper is to show that this collapsing behavior on surfaces of Kodaira dimension one in the examples of \cite{TW2} actually occurs for \emph{every choice} of initial starting Gauduchon metric $\omega_0$.  Combined with Gill's theorem, our results mean that the only remaining case (presumably the most difficult!) under assumption $(*)$ is to understand the behavior of the Chern-Ricci flow on surfaces of negative Kodaira dimension.
We believe that the results of this paper add to the growing body of evidence that the Chern-Ricci flow is a natural geometric evolution equation on complex surfaces, whose behavior reflects the underlying geometry of the manifold.

We first consider the case of elliptic bundles over a Riemann surface.  Later we will see that this is sufficient to understand the behavior of the flow on all $M$ satisfying $(*)$ with $\Kod(M)=1$.
 Suppose that $\pi:M\to S$ is now an elliptic bundle over a compact Riemann surface $S$ of genus at least $2$, with fiber an elliptic curve $E$.
We will denote by $E_y=\pi^{-1}(y)$ the fiber over a point $y\in S$ and by  $\omega_{\mathrm{flat},y}$
the unique flat metric on $E_y$ in the K\"ahler class $[\omega_0|_{E_y}]$.
Let $\omega_S$ be the unique  K\"ahler-Einstein metric on $S$ with $\Ric(\omega_S)=-\omega_S$ and let $\omega_0$ be a Gauduchon metric on $M$.

We consider the \emph{normalized Chern-Ricci flow}
\begin{equation}  \label{NCRF}
\frac{\de}{\de t}\omega=-\Ric(\omega)-\omega, \qquad \omega|_{t=0} =\omega_0,
\end{equation}
starting at $\omega_0$.  With this normalized flow we will see that the volume of the base Riemann surface $S$ remains positive and bounded while the elliptic fibers collapse.
One could equally well study the unnormalized flow (\ref{crf0}) on $M$ (so that our main collapsing result would apply to $\omega(t)/t$ as in \cite{TW2}) but we choose this normalization to stay in keeping with the literature on the K\"ahler-Ricci flow \cite{ST}.  From \cite{TW} we know that a smooth solution to \eqref{NCRF} exists for all time (see Section \ref{section:prel} below for more details).  In this paper we prove the following convergence result as $t \rightarrow \infty$.

\begin{theorem}\label{main}  Let $\pi: M \rightarrow S$ be an elliptic bundle over a Riemann surface $S$ of genus at least 2.
Let $\omega(t)$ be a solution  of the normalized Chern-Ricci flow
\eqref{NCRF} on $M$ starting at a Gauduchon metric $\omega_0$.  Then as $t \rightarrow \infty$,
$$\omega(t)\to \pi^*\omega_{S},$$
exponentially fast in the $C^0(M,g_0)$ topology, where $\omega_S$ is the unique K\"ahler-Einstein metric on $S$.
In particular, the diameter of each elliptic fiber tends to zero uniformly exponentially fast and
 $(M,\omega(t))$ converges to $(S,\omega_S)$ in the Gromov-Hausdorff topology.

Furthermore, with the notation above, $e^{t}\omega(t)|_{E_y}$ converges to the metric
$\omega_{\mathrm{flat},y}$ exponentially fast in the $C^1(E_y, g_0)$ topology, uniformly in $y\in S$.
\end{theorem}

Note that in Theorem \ref{main} we do not need to assume that $M$ is non-K\"ahler.  On the other hand, we do assume that $M$ is an elliptic bundle, so that the fibers are all isomorphic as elliptic curves.  General elliptic surfaces may have singular fibers and in such cases,  the complex structure of the smooth  fibers may vary.   However, we will see shortly that this does not arise for the non-K\"ahler surfaces that are of interest to us.

  In the case that $M$ is K\"ahler \emph{and} $\omega_0$ is K\"ahler, then $\omega(t)$ is a solution of  the normalized K\"ahler-Ricci flow.  There are already a number of results on this, which we now briefly discuss.
On a general minimal K\"ahler elliptic surface, and its higher dimensional analogue, the K\"ahler-Ricci flow was first investigated by Song-Tian \cite{ST, ST2}.  They showed that the flow converges at the level of potentials to a \emph{generalized K\"ahler-Einstein metric} on the base Riemann surface.  The generalized K\"ahler-Einstein equation involves the Weil-Petersson metric and singular currents.  These terms arise because, unlike in our case, the fibration structure on a K\"ahler elliptic surface is not in general locally trivial and may have singular fibers.  When the K\"ahler surface is a genuine elliptic bundle over a Riemann surface of genus larger than one, the results of Song-Tian give $C^0$ collapsing of the fibers along the K\"ahler-Ricci flow, as well as a uniform scalar curvature bound \cite{ST3}.  These convergence results were strengthened  by Song-Weinkove \cite{SW} and Gill \cite{G2} in the special case of a product $E \times S$, giving $C^{\infty}$ convergence of the metrics to the pull-back of a K\"ahler-Einstein metric on the base.  Fong-Zhang \cite{FZ}, adapting a technique of Gross-Tosatti-Zhang \cite{GTZ} on Calabi-Yau degenerations, established smooth convergence for the K\"ahler-Ricci flow on more general elliptic bundles.  In particular, the statement of Theorem \ref{main} is known if the initial metric $\omega_0$ is K\"ahler (with the exception of the assertion that the convergence $\omega(t) \rightarrow \pi^* \omega_S$ is exponential - as far as we know, this result is new even in the K\"ahler-Ricci flow case).

Of course, we are much more interested in manifolds which do \emph{not} admit K\"ahler metrics.  For non-K\"ahler elliptic surfaces, we make use of the following {\bf key fact}:

\bigskip
\emph{Every minimal non-K\"ahler properly elliptic surface is an elliptic bundle or has a finite cover which is an elliptic bundle.}
 \bigskip

 This  is well-known from the Kodaira classification (see for example \cite[Lemmas 1, 2]{Br2} or \cite[Theorem 7.4]{Wa}).  Then an immediate consequence of our Theorem \ref{main} is that we can identify the Gromov-Hausdorff behavior of the Chern-Ricci flow on all minimal non-K\"ahler surfaces of Kodaira dimension one.

\begin{corollary}\label{cor} Let  $\pi:M\to S$ be any minimal non-K\"ahler properly elliptic surface and let $\omega(t)$ be the solution of the normalized Chern-Ricci flow (\ref{NCRF}) starting at a Gauduchon metric $\omega_0$. 
Then $(M,\omega(t))$ converges to $(S,d_S)$ in the Gromov-Hausdorff topology.

Here $d_S$ is the distance function induced by an orbifold K\"ahler-Einstein metric $\omega_S$ on $S$, whose set $Z$ of orbifold points is precisely the image of the multiple fibers of $\pi$. Furthermore, $\omega(t)$ converges to $\pi^*\omega_S$ in the $C^0(M,g_0)$ topology, and for any $y\in S\backslash Z$ the metrics $e^{t}\omega(t)|_{E_y}$
converge exponentially fast in the $C^1(E_y, g_0)$ topology (and uniformly as $y$ varies in a compact set of $S\backslash Z$) to the flat K\"ahler metric on $E_y$ cohomologous to $[\omega_0|_{E_y}]$.
\end{corollary}

As mentioned above, explicit examples exhibiting the behavior of Corollary \ref{cor} were given in \cite{TW2}.

We now outline the steps we need to establish Theorem \ref{main}, and point out some of the difficulties that arise from the non-K\"ahlerity of the metrics.

The first parts of the proof follow quite closely the arguments used by Song-Tian \cite{ST} for the K\"ahler-Ricci flow.
In Section \ref{section:prel}, we show that the Chern-Ricci flow can be written as parabolic complex Monge-Amp\`ere equation
\begin{equation} \label{pcma}
\frac{\de}{\de t}\vp=\log\frac{e^t(\tilde{\omega}+\ddbar\vp)^2}{\Omega}-\vp,\  \tilde{\omega}+\ddbar\vp>0,\  \vp(0)=-\rho,
\end{equation}
where $\tilde{\omega}=\tilde{\omega}(t)$ is a family of reference forms (which are metrics for $t$ large) given by $$\tilde{\omega}=e^{-t}\of + (1-e^{-t}) \omega_S, \quad \textrm{with } \of = \omega_0 + \ddbar \rho,$$
where $\rho$ is chosen so that $\of$ restricted to the fiber $E_y$ is exactly the metric $\omega_{\textrm{flat}, y}$ discussed above.  Here $\Omega$ is a particular fixed volume form on $M$ with the property that $\ddbar \log \Omega = \omega_S$.  If $\varphi$ satisfies (\ref{pcma}) then $\omega(t) = \tilde{\omega}+\ddbar \varphi$ satisfies the Chern-Ricci flow (\ref{NCRF}).

In Section 3 we establish  uniform bounds for $\varphi$ and $\dot{\varphi}$, which imply that the volume form of the evolving metric $\omega(t)$ is uniformly equivalent to the volume form of the reference metric.  These follow in the same way as in the case of the K\"ahler-Ricci flow \cite{ST}.  In addition, we prove a crucial decay estimate for $\varphi$,
\begin{equation} \label{eqndecay}
| \varphi | \le C(1+t)e^{-t},
\end{equation}
using the argument of \cite{SW}.  This estimate makes use of the Gauduchon assumption on $\omega_0$, and in fact is the only place where we use this condition.

So far, the torsion terms of $\omega(t)$ and $\tilde{\omega}$  have not
entered the picture.  They show up in the next step of obtaining uniform bounds for the metrics $\omega(t)$.  The evolution equation for $\tr{\tilde{\omega}}{\omega}$, essentially already computed in \cite{TW}, contains terms involving the torsion and curvature of the reference metrics $\tilde{\omega}$.  In Section \ref{sectorsion} we prove a technical lemma giving bounds for the torsion and curvature of these metrics.  In particular we show:
$$| \tilde{T}|_{\tilde{g}} \le C, \quad | \db \tilde{T}|_{\tilde{g}} + | \ti{\nabla} \tilde{T}|_{\tilde{g}} +  | \widetilde{\textrm{Rm}} |_{\tilde{g}}  \le C e^{t/2}.$$
To deal with these bounds of order $e^{t/2}$, our idea is to exploit the strong decay estimate (\ref{eqndecay}) on $\varphi$ to control these terms.

 In Section \ref{sectiontrace}, we evolve the quantity
 $$Q=\log \tr{\tilde{\omega}}{\omega} - Ae^{t/2} \varphi + \frac{1}{\tilde{C}+e^{t/2} \varphi},$$
 noting that $e^{t/2}\varphi$ is bounded by (\ref{eqndecay}).  The third term of $Q$ is the ``Phong-Sturm term'' \cite{PS}, which was used in \cite{TW} to control some torsion terms along the Chern-Ricci flow.  Using the good positive terms arising from the Laplacian landing on $e^{t/2}\varphi$ we can control the bad terms of order $e^{t/2}$ coming from the torsion and curvature of $\tilde{\omega}$.
  We obtain a uniform bound on $Q$ which gives the estimate
 \begin{equation} \label{ooequiv}
 C^{-1} \tilde{\omega} \le \omega \le C \tilde{\omega},
 \end{equation}
 namely, that the solution $\omega$ is uniformly equivalent to the reference metric $\tilde{\omega}$.

 We point out that our argument here differs substantially from that of Song-Tian \cite{ST} where they prove first a parabolic Schwarz Lemma, namely an estimate of the type $\omega \ge C^{-1} \omega_S$ for a uniform $C>0$.  We were unable to prove this by a similar direct maximum principle argument, because of troublesome torsion terms arising in the evolution of $\tr{\omega}{\omega_S}$.  However, we still obtain the estimate $\omega \ge C^{-1} \omega_S$ once we have (\ref{ooequiv}).

 The next step is to improve the bound (\ref{ooequiv}) to the stronger exponential convergence result
 \begin{equation} \label{expbdcrf}
 (1- Ce^{-\ve t}) \tilde{\omega} \le \omega \le (1+Ce^{-\ve t}) \tilde{\omega},
\end{equation}
for $\ve>0$. To our knowledge, this estimate is new even for the K\"ahler-Ricci flow on elliptic bundles.  The idea is to evolve the quantity
$$Q = e^{\ve t} (\textrm{tr}_{\omega}{\tilde{\omega}} - 2) - e^{\delta t} \varphi,$$
for a carefully chosen $\delta >1/2+2\ve$   and again exploit the decay estimate (\ref{eqndecay}).  Showing that $Q$ is bounded from above then gives the estimate
$$\tr{\omega}{\tilde{\omega}} -2\le C e^{-\ve t},$$
 and a similar argument gives the same inequality with $\tr{\omega}{\tilde{\omega}}$ replaced by $\tr{\tilde{\omega}}{\omega}$.  Combining these two estimates gives (\ref{expbdcrf}).

However, in order to apply the maximum principle to $Q$ we first require an exponential decay estimate for $\dot{\varphi}$.
To prove this, we observe that the evolution equation for $\dot{\varphi}$ is
$$\ddt{} \dot{\varphi} = - R-1-\dot{\varphi},$$
where $R$ is the Chern scalar curvature of $g$.  If we had a uniform bound for the Chern scalar curvature, an exponential decay estimate for $\dot{\varphi}$ would follow from this evolution equation and the decay estimate for $\varphi$.  However, we are only able to prove the weaker estimate
\begin{equation} \label{introscalar}
-C \le R \le Ce^{t/2}.
\end{equation}
Nevertheless, this suffices since the coefficient of $t$ in the exponent is strictly less than 1.
The bound (\ref{introscalar}) is the content of Section \ref{sectionscalar}.  The factor $e^{t/2}$ arises from the bounds on the curvature and torsion of the reference metrics we obtained in Section \ref{sectorsion}.

The idea for bounding the Chern scalar curvature from above (the lower bound is easy) is to consider the quantity $u=\varphi + \dot{\varphi}$ and bound from above $-\Delta u = R +\tr{\omega}{\omega_S} \ge R$.  Using an idea that goes back to Cheng-Yau \cite{CY}, and is used in the context of the K\"ahler-Ricci flow on Fano manifolds by Perelman (see Sesum-Tian \cite{SeT}) and on elliptic surfaces by Song-Tian \cite{ST}, we first bound the gradient of $u$ by considering the quantity $|\nabla u|^2_g/(A-u)$ for a fixed large $A$.  We then evolve the quantity $-\Delta u + 6 | \nabla u|^2_g$, which is almost enough to obtain the estimate we need.  There are some bad terms which we can control by adding large multiples of $\tr{\omega}{\omega_S}$ and $\tr{\tilde{\omega}}{\omega}$.  We already know that these terms are bounded from (\ref{ooequiv}).  Putting this together gives the upper bound on scalar curvature and the exponential decay estimate for $\dot{\varphi}$ that we require.

Now that we have this exponential decay estimate on $\dot{\varphi}$, we carry out
in Section \ref{sectionexp} the argument mentioned above for the exponential convergence of the metrics  (\ref{expbdcrf}).

In Section \ref{sectioncalabi}, we prove a local Calabi type estimate
\begin{equation} \label{localcalabiestimate}
| \hat{\nabla} g|^2_g \le Ce^{2t/3},
\end{equation}
where $\hat{\nabla}$ is the connection associated to a local semi-flat product K\"ahler metric defined in a neighborhood $U$.  Note that if we had the better  estimate  $| \hat{\nabla} g|^2_g \le C$ (as in  \cite{SW, G2, FZ} for example) then we could immediately conclude  the global convergence of the metrics $\omega(t)$ to $\pi^* \omega_S$ from the estimates (\ref{eqndecay}) and (\ref{ooequiv}) and the Ascoli-Arzel\`a Theorem.  We do not know whether this stronger estimate $| \hat{\nabla} g|^2_g \le C$  holds or not.

To establish (\ref{localcalabiestimate}), we use some arguments and calculations similar to the local Calabi estimate in  \cite{ShW}.  However, a key difference here is that the metrics are collapsing in the fiber directions and we need to take account of the error terms that arise in this way. The local Calabi estimate is then used to establish the last part of Theorem \ref{main} that $e^{t}\omega(t)|_{E_y}$ converges to
$\omega_{\mathrm{flat},y}$ exponentially fast in the $C^1(E_y, g_0)$ topology, uniformly in $y\in S$.

In Section \ref{sectionproofs} we complete the proofs of Theorem \ref{main} (this essentially follows immediately) and Corollary \ref{cor}.

\section{Preliminaries}\label{section:prel}

\subsection{Hermitian geometry and notation}

We begin with a brief recap of  Hermitian geometry and the Chern connection (for more details see for example \cite{TW}).

Given a Hermitian metric $g$, we write $\omega = \sqrt{-1} g_{i\ov{j}} dz^i \wedge d\ov{z}^j$ for its associated $(1,1)$ form, which we will also refer to as a metric.  Write
 $\nabla$ for its Chern connection, with respect to which $g$ and the complex structure are covariantly constant.  The Christoffel symbols of $\nabla$ are given by $\Gamma_{ij}^k=g^{k\ov{q}}\partial_i g_{j\ov{q}}$.  For example, if $X=X^i\partial_i$ is a vector field then its covariant derivative has components $\nabla_i X^{\ell} = \partial_i X^{\ell} + \Gamma^{\ell}_{i j} X^j$.

The torsion tensor of $g$ has components $T^k_{ij} = \Gamma^k_{ij} - \Gamma^k_{ji}$.  We will often lower an index using the metric $g$, writing $$T_{ij\ov{\ell}} = g_{k\ov{\ell}} T^k_{ij} = \partial_i g_{j\ov{\ell}} - \partial_j g_{i\ov{\ell}}.$$
Note that $T_{ij\ov{\ell}} = T'_{ij\ov{\ell}}$ if $g$ and $g'$ are Hermitian metrics whose $(1,1)$ forms $\omega$ and $\omega'$ differ by a closed form.
 The Chern curvature of $g$ is defined to be $R_{k\ov{\ell} i}^{\ \ \ \, p} = -\partial_{\ov{\ell}} \Gamma^p_{ki}$, and we will raise and lower indices using the metric $g$.  We have the usual commutation formulae involving the curvature, such as
$$[\nabla_k, \nabla_{\ov{\ell}}]X^i = R_{k\ov{\ell}j}^{\ \ \ \, i} X^j.$$
Define the Chern-Ricci curvature of $g$ to be $R_{k\ov{\ell}} = g^{i\ov{j}} R_{k\ov{\ell}i\ov{j}} = -\partial_k \partial_{\ov{\ell}} \log \det g$, and we write
$$\textrm{Ric}(\omega)=\sqrt{-1} R_{k\ov{\ell}} dz^k \wedge d\ov{z}^{\ell}$$
for the associated Chern-Ricci form, a real closed (1,1) form. Write $R=g^{k\ov{\ell}} R_{k\ov{\ell}}$ for the Chern scalar curvature.

We write $\Delta$ for the complex Laplacian of $g$, which acts on a function $f$ by $\Delta f=g^{i\ov{j}} \partial_i \partial_{\ov{j}}f$.   For functions $f_1, f_2$, we define $\langle \nabla f_1, \nabla f_2 \rangle_g = g^{i\ov{j}} \partial_if_1 \partial_{\ov{j}} f_2$ and $|\nabla f|^2_g = \langle \nabla f, \nabla f\rangle_g$.   If $\alpha= \sqrt{-1} \alpha_{i\ov{j}}dz^i \wedge d\ov{z}^j$ is a real $(1,1)$ form and $\omega$ a Hermitian metric we write $\tr{\omega}{\alpha}$ for $g^{i\ov{j}} \alpha_{i\ov{j}}$.

A final remark about notation:  we will write $C, C', C_0, \ldots$ etc. for a uniform constant, which may differ from line to line.

\subsection{Elliptic bundles  and semi-flat metrics}
We now specialize to the setting of Theorem \ref{main}.  Let $\pi:M\to S$ be an elliptic bundle over a compact Riemann surface $S$ of genus at least $2$, with fiber an elliptic curve $E$.
Clearly $\pi:M\to S$ is relatively minimal, because there is no $(-1)$-curve contained in any fiber.
We will denote by $E_y=\pi^{-1}(y)$ the fiber over a point $y\in S$.
Let $\omega_S$ be the unique  K\"ahler metric on $S$ with $\Ric(\omega_S)=-\omega_S$, let $\omega_0$ be a Gauduchon metric on $M$.

Since each fiber $E_y=\pi^{-1}(y)$ is a torus, we can find a function $\rho_y$ on $E_y$ with
$$\omega_0|_{E_y}+\ddbar\rho_y = \omega_{\mathrm{flat},y},$$
the unique flat metric on $E_y$ in the K\"ahler class $[\omega_0|_{E_y}]$.
Furthermore we can normalize the functions $\rho_y$ by $\int_{E_y}\rho_y \omega_0=0$,
so that they vary smoothly in $y$ (in general this follows from Yau's estimates \cite{Y1},
although in this simple case it can also be proved directly, see also \cite[Lemma 2.1]{Fi}), and they define a smooth function $\rho$ on $M$.
We then let
\begin{equation} \label{defnof}
\omega_{\mathrm{flat}}=\omega_0+\ddbar\rho.
\end{equation}
$\omega_{\mathrm{flat}}$ is a semi-flat form, in the sense that it restricts to a flat metric on each fiber $E_y$,
but in general it is not positive definite on $M$. But note that $\omega_{\mathrm{flat}}\wedge\pi^*\omega_S$ is a strictly positive
smooth volume form on $M$.

\subsection{The canonical bundle and long time existence for the flow}
In the same setting as above, we claim that $K_M=\pi^* K_S$. To see this, start from Kodaira's canonical bundle formula for relatively minimal elliptic surfaces
without singular fibers \cite[Theorem V.12.1]{bhpv}
$$K_M=\pi^*(K_S\otimes L),$$
where $L$ is the dual of $R^1\pi_*\mathcal{O}_M$. But since $M$ is an elliptic bundle, it follows that
the line bundle $R^1\pi_*\mathcal{O}_M$ is trivial (see e.g. \cite[Proposition 2.1]{Br1}), and the claim follows.

Therefore $c_1(M)=\pi^* c_1(S)$ (an alternative more direct proof of this fact is contained in Lemma \ref{good}), and so there exists a unique volume form $\Omega$ with
\begin{equation} \label{Omdefn}
\Ric(\Omega)=-\omega_S \quad \textrm{and} \quad \int_M \Omega= 2\int_M \omega_0\wedge \omega_S.
\end{equation}
Here and henceforth, we are abbreviating $\pi^*\omega_S$ by $\omega_S$, and for any smooth positive volume form $\Omega$ we write
$\Ric(\Omega)$ for the globally defined real $(1,1)$-form given locally by $-\ddbar \log \Omega$.

It follows that the Bott-Chern class of the canonical bundle $K_M$, which equals $c_1^{\mathrm{BC}}(K_M)=-c_1^{\mathrm{BC}}(M),$ is {\em nef}. In general this means that given any $\ve>0$ there
exists a real smooth function $f_\ve$ on $M$ such that $-\Ric(\omega_0)+\ddbar f_\ve > -\ve\omega_0$. Equivalently, this can be phrased by saying that
for any $\ve>0$ there is a smooth Hermitian metric $h_\ve$ on the fibers of $K_M$ with curvature form bigger than $-\ve\omega_0.$ The maximal existence theorem for the Chern-Ricci flow
\cite[Theorem 1.2]{TW} has  the following immediate corollary:

\begin{theorem}
Let $(M,\omega_0)$ be any compact Hermitian manifold. Then the Chern-Ricci flow
\begin{equation}  \label{UCRF}
\frac{\de}{\de t}\omega=-\Ric(\omega), \qquad \omega|_{t=0} =\omega_0,
\end{equation}
has a smooth solution defined for all $t\geq 0$ if and only if the first Bott-Chern class $c_1^{\mathrm{BC}}(K_M)$ is nef.
The exact same statement holds for the normalized Chern-Ricci flow \eqref{NCRF}.
\end{theorem}
Since this theorem was not stated explicitly in \cite{TW}, we provide the simple proof.
\begin{proof}
An elementary space-time scaling argument \cite{Ha} allows one to transform a solution of \eqref{NCRF} into a solution of \eqref{UCRF} and vice versa,
and one exists for all positive time if and only if the other one does, so it is enough to consider \eqref{UCRF}.

In this case, we know from \cite{TW} that as long as a solution $\omega(t)$ exists, it is of the form
$$\omega(t)=\omega_0-t\Ric(\omega_0)+\ddbar \vp(t),$$
and therefore
$$-\Ric(\omega_0)+\ddbar \left(\frac{\vp}{t}\right) > -\frac{1}{t}\omega_0,$$
so if the solution exists for all $t\geq 0$, then we see that $c_1^{\mathrm{BC}}(K_M)$ is nef.

Conversely, if $c_1^{\mathrm{BC}}(K_M)$ is nef, then for every given $t>0$ we can find a smooth function $f_t$ with
$$-\Ric(\omega_0)+\ddbar f_t > -\frac{1}{t}\omega_0,$$
which is equivalent to
$$\omega_0-t\Ric(\omega_0)+\ddbar (tf_t)>0,$$
and so the flow exists at least on $[0,t)$ by \cite[Theorem 1.2]{TW}.
\end{proof}

Applying this to the setting of Theorem \ref{main}, we obtain a smooth solution $\omega(t)$ to the normalized Chern-Ricci flow (\ref{NCRF}) for $t \in [0,\infty)$.

\subsection{The  parabolic complex Monge-Amp\`ere equation}
From now on, until we get to Section \ref{sectionproofs}, we assume we are in the setting of Theorem \ref{main}.
We will rewrite the normalized Chern-Ricci flow \eqref{NCRF} as a parabolic complex Monge-Amp\`ere equation.  Define reference $(1,1)$-forms $\tilde{\omega}=\tilde{\omega}(t)$ by
$$\ti{\omega}=e^{-t}\omega_{\mathrm{flat}}+(1-e^{-t})\omega_S,$$
where we recall that $\of$ is defined by (\ref{defnof}).
Note that $\tilde{\omega}$ may not necessarily be positive definite for all $t$, but there exists a time $T_I$ such that $\tilde{\omega}>0$ for all $t \ge T_I$.  (On the other hand, observe that
$\ti{\omega}-e^{-t}\ddbar\rho$ is positive definite for all $t\geq 0$).    We fix this constant $T_I$ now once and for all. By the long time existence result of \cite{TW}, we have uniform $C^{\infty}$ estimates on $\omega(t)$  for $t \in [0, T_I]$.  Our goal is to obtain estimates on $\omega(t)$  for $t > T_I$ which are independent of $t$.

Define a function $\vp(t)$ by
$$\frac{\de}{\de t}\vp=\log\frac{e^t\omega(t)^2}{\Omega}-\vp,\quad \vp(0)= -\rho,$$
where we recall that $\omega_{\mathrm{flat}}=\omega_0+\ddbar\rho$ and $\Omega$ is given by (\ref{Omdefn}).
We claim that $\omega(t)=\tilde{\omega}+\ddbar\vp(t)$ holds. Indeed,
$$\frac{\de}{\de t}\tilde{\omega}=\omega_S-\tilde{\omega}=-\Ric(\Omega)-\tilde{\omega},$$
and so,
$$\frac{\de}{\de t}(e^t(\omega-\tilde{\omega}-\ddbar\vp))=0,\quad (e^t(\omega-\tilde{\omega}-\ddbar\vp))|_{t=0}=0,$$
which implies that indeed $\omega=\tilde{\omega}+\ddbar\vp.$ Therefore $\vp$ also satisfies the PDE
\begin{equation}\label{CRF}
\frac{\de}{\de t}\vp=\log\frac{e^t(\tilde{\omega}+\ddbar\vp)^2}{\Omega}-\vp,\  \tilde{\omega}+\ddbar\vp>0,\  \vp(0)=-\rho,
\end{equation}
and conversely every solution of \eqref{CRF} gives rise to a solution $\omega=\tilde{\omega}+\ddbar\vp$ of the normalized Chern-Ricci flow \eqref{NCRF}.

\section{Estimates on the potential and its time derivative}

We now begin the proof of Theorem \ref{main}. We assume $\pi: M \rightarrow S$ is an elliptic bundle over a Riemann surface $S$ of genus at least $2$ and $\omega(t)$ is a solution of the normalized Chern-Ricci flow (\ref{NCRF}) on $M$ starting at a Gauduchon metric $\omega_0$.

In this section we collect some estimates on the potential function $\varphi$ solving  \eqref{CRF}, and its time derivative $\dot{\varphi}:= \partial \varphi/\partial t$.  
The proofs of these results are almost identical to the corresponding results for the K\"ahler-Ricci flow \cite{ST} (see also \cite{SW}).  For the reader's convenience we include here the brief arguments. We also point out the one place where we use the Gauduchon condition.

\begin{lemma}\label{lowerbound} There exists a uniform positive constant $C$ such
that on $M$,
\begin{enumerate}
\item[(i)]  $\displaystyle{|\varphi(t)|\leq C}$ for all $t \ge 0$.
\item[(ii)]  $\displaystyle{|\dot \varphi(t)|\leq C}$ for all $t \ge 0$.
\item[(iii)] $\displaystyle{\frac{1}{C}\tilde \omega^2\leq \omega^2\leq C\tilde\omega^2}$ for all $t \ge T_I$.
\end{enumerate}

\bproof We follow the exposition in \cite{SW}. Since $e^t \tilde
\omega^2=e^{-t}\of^2+2(1-e^{-t})\of\wedge \omega_S$, we
have for $ t\ge T_I$,
\begin{equation} \label{volformbds}
\frac{1}{C}\Om\leq e^t \tilde \omega^2\leq C\Om.
\end{equation}
If $\varphi$ attains a minimum at a point $(x_0,t_0)$ with $t_0>T_I$, then
at that point
$$0\geq \ddt \varphi= \log\frac{ e^t(\tilde\omega+\ddbar \varphi)^2}{\Om}-\varphi\geq\log\frac{ e^t\tilde\omega^2}{\Om}-\varphi\geq -\log C-\varphi,$$
giving $\varphi\geq -\log C$ and hence a uniform lower bound for $\varphi$. The upper bound for $\varphi$ is similar. This gives (i).

For (ii),  we first compute:
\begin{equation} \label{eveqndotphi}
\left(\ddt-\Delta\right)\dot\varphi= \tr{\omega}{(\omega_S-\tilde\omega)}+1-\dot\varphi.
\end{equation}
On the other hand, there exists a uniform constant
$C_0>1$ such that $C_0\tilde\omega>\omega_S$ for $t \ge T_I$.
We apply the maximum principle
to $Q_1=\dot\varphi-(C_0-1)\varphi$.  Calculate for $t \ge T_I$,
\begin{align*}
\left(\ddt -\Delta\right)Q_1={} & \tr{\omega}{(\omega_S-\tilde\omega)}+1-C_0\dot\varphi+(C_0-1)\tr{\omega}{(\omega-\tilde\omega)}\\
\leq {} & 1-C_0\dot\varphi+2(C_0-1),
\end{align*}
and the maximum principle shows that $Q_1$ is bounded from above uniformly.  This gives the upper bound for $\dot{\varphi}$.

Next consider
 $Q_2=\dot\varphi+2\varphi$ and compute
$$\left(\ddt -\Delta\right)Q_2=\tr{\omega}{(\omega_S-\tilde\omega)}+1+\dot\varphi-2\tr{\omega}{(\omega-\tilde\omega)} \geq \tr{\omega}{\tilde\omega}+\dot\varphi-3.$$
By  the geometric-arithmetic means inequality, we have for $t\geq T_I$,
\begin{equation}
e^{-\frac{\dot\varphi+\varphi}{2}}=\left(\frac{\Om}{e^t\omega^2}\right)^{\frac{1}{2}}\leq
C\left(\frac{\tilde\omega^2}{\omega^2}\right)^{\frac{1}{2}}\leq
\frac{C}{2} \tr{\omega}{\tilde\omega}.\label{lb}
\end{equation}
Then at a point $(x_0, t_0)$ with $t_0>T_I$ where
 $Q_2$ attains a minimum, $\tr{\omega}{\tilde
\omega}\leq 3-\dot\varphi$ and so
$e^{-\frac{\dot\varphi+\varphi}{2}}\leq C(3-\dot\varphi),$
which gives a uniform lower bound for $\dot{\varphi}$.  This completes the proof of (ii).

Part (iii) follows from (i) and (ii) and the equations \eqref{CRF} and \eqref{volformbds}.
 \end{proof}
\end{lemma}

Our next result is an exponential decay estimate for $\varphi$.  We first need a lemma.  Recall that the volume form $\Omega$ is defined by \eqref{Omdefn}.  This lemma is the only place in the paper where we make use of the Gauduchon assumption on $\omega_0$.

\begin{lemma}\label{good} We have that
\begin{equation}\label{goal}
\Omega=2\omega_{\mathrm{flat}}\wedge\omega_S.
\end{equation}
\end{lemma}
\begin{proof}
Since $2\int_M \omega_{\mathrm{flat}}\wedge\omega_S=2\int_M \omega_{0}\wedge\omega_S=\int_M \Omega$, it is enough to show that
\begin{equation}\label{goal2}
\ddbar\log \frac{\Omega}{\omega_{\mathrm{flat}}\wedge\omega_S}=0.
\end{equation}
Recalling that $M$ is an elliptic bundle with fiber $E$, we can fix a small ball $B\subset S$ over which $\pi$ is holomorphically trivial, so $\pi^{-1}(B)\cong B\times E$.
If we identify $E=\mathbb{C}/\Lambda$, for some lattice $\Lambda\subset\mathbb{C}$, and call $z^1$ the coordinate
on $\mathbb{C}$, then $dz^1$ descends to a never vanishing holomorphic $1$-form on $E$.
If we call $\alpha$ its pullback to $B\times E$, then $\mn\alpha\wedge\ov{\alpha}$ is a smooth semi-flat form
on $\pi^{-1}(U)$. Then there is a function $u(y)$ defined on $B$ such that for any $y\in B$ we have
$$\omega_{\mathrm{flat}}|_{E_y}=u(y) \mn\alpha\wedge\ov{\alpha}.$$
This is because both $\omega_{\mathrm{flat}}|_{E_y}$ and $\mn\alpha\wedge\ov{\alpha}$ are flat volume forms on $E_y$, and so their ratio is a constant on $E_y$.
Integrating this equality over $E_y$ we get
$$\int_{E_y}\omega_{\mathrm{flat}}=u(y)\int_{E_y} \mn\alpha\wedge\ov{\alpha}.$$
But on the one hand the integral $\int_{E_y} \mn\alpha\wedge\ov{\alpha}$ is independent of $y$ by definition, and on the other hand
the function $y\mapsto \int_{E_y}\omega_{\mathrm{flat}}$ is also constant in $y$, because it equals the pushforward $\pi_*\omega_{\mathrm{flat}}$ and we have
$$\de\db \pi_*\omega_{\mathrm{flat}} = \pi_*\de\db\omega_{\mathrm{flat}}=\pi_*\de\db\omega_0=0.$$
Note that the last equality uses the Gauduchon condition.  This
 implies that $\pi_*\omega_{\mathrm{flat}}$ is constant by the strong maximum principle. Therefore $u$ is a constant.

Fix now a point $x\in M$, call $y=\pi(x)$, and choose local bundle coordinates near $x$ and $y$, so that in these coordinates
the projection $\pi$ is given by $\pi(z^1,z^2)=z^2$. Then write locally
$$\omega_S = \mn g(z^2) dz^2\wedge d\ov{z}^2,$$
$$\Omega=G(z^1,z^2) (\mn)^2 dz^1\wedge d\ov{z}^1\wedge dz^2\wedge d\ov{z}^2,$$
and compute
\begin{equation}\label{calc}
F:=\frac{\Omega}{\omega_{\mathrm{flat}}\wedge\omega_S}=\frac{\Omega}{u\sqrt{-1}\alpha\wedge\ov{\alpha}\wedge\omega_S}=G(ug)^{-1},
\end{equation}
and so locally on $S$ we have
\begin{equation}\label{crux2}
\ddbar\log F=\omega_S+\Ric(\omega_S)=0,
\end{equation}
because $\Ric(\Omega)=-\omega_S=\Ric(\omega_S)$. This proves \eqref{goal2}.

Incidentally the same calculation proves that the volume form $\omega_{\mathrm{flat}}\wedge\omega_S$ satisfies
$$\Ric(\omega_{\mathrm{flat}}\wedge\omega_S)=- \omega_S,$$
which gives another proof of the fact that $c_1(M)=\pi^*c_1(S)$.
\end{proof}

Very similar arguments can be found in the paper of Song-Tian \cite{ST}, in the K\"ahler case (see also \cite{ST2, To}).\\

\begin{remark} Lemma \ref{good} fails if we drop the assumption that $\de\db\omega_0=0$. Indeed, consider the case when $M=S\times E$, and
let $\omega_E$ be a flat K\"ahler metric on $E$ (while $\omega_S$ is as before).
If $F:S\to\mathbb{R}$ is any nonconstant positive function then
$$\omega_0=F\omega_E+\omega_S,$$
is a Hermitian metric on $M$ with $\de\db\omega_0\neq 0$. We have that $\omega_0=\omega_{\mathrm{flat}}$, because $\omega_0$ is already semi-flat.
On the other hand, $\Omega = c\cdot\omega_E\wedge\omega_S$, where $c$ is the constant given by
$$c=\frac{2\int_M \omega_0\wedge\omega_S}{\int_M \omega_E\wedge\omega_S}.$$
Therefore we have
$$\Omega\neq 2\omega_{\mathrm{flat}}\wedge\omega_S=F\omega_E\wedge\omega_S.$$
\end{remark}

We can now prove a decay estimate for $\varphi$.  The analogous estimate was proved for the K\"ahler-Ricci flow on a product surface in \cite[Lemma 6.7]{SW}.  The proof in this case is almost identical, given Lemma \ref{good}.

\begin{lemma}\label{key}  There exists a uniform constant $C>0$
such that on $M\times [0,\infty)$,
$$|\varphi|\leq C(1+t)e^{-t}.$$
\end{lemma}
\begin{proof}
First, we claim that for $t\geq T_I$,
\begin{equation} \label{claimOmega}
\left| e^t  \log \frac{e^t \tilde{\omega}^2}{\Omega}\right|\le C.
\end{equation}
This follows from  the argument in \cite[Lemma 6.7]{SW}.  Indeed,  using Lemma \ref{good}, we see that
\begin{align*}
e^t  \log \frac{e^t \tilde{\omega}^2}{\Omega} = {} & e^t \log \frac{2\omega_S \wedge \of + e^{-t} (\of^2 - 2\of \wedge \omega_S)}{2\omega_S \wedge \of} \\
= {} & e^t \log (1+ O(e^{-t})),
\end{align*}
which is bounded.

Define now $Q=e^t \varphi +At$, for $A$ a large positive constant to be determined.   Then
\begin{equation}
\frac{\partial Q}{\partial t} = e^t \log \left( \frac{e^t (\tilde{\omega}+ \ddbar \varphi)^2}{\Omega} \right)+ A.
\end{equation}
We wish to bound $Q$ from below.  Suppose that $(x_0, t_0)$ is a point with $t_0 > T_I$ at which $Q$ achieves a minimum.  At this point we have
\begin{align*}
0 \ge \frac{\partial Q}{\partial t} \ge {} &
e^t\log\frac{e^t\ti{\omega}^2}{\Omega}+A \ge   -C'+A,
\end{align*}
for a uniform $C'$, thanks to \eqref{claimOmega}.
 Choosing $A> C'$ gives contradiction.
Hence $Q$ is bounded from below and it follows that $\varphi\geq -C(1+t)e^{-t}$ for a uniform
$C$. The upper bound for $\varphi$ is similar.
 \end{proof}

\section{Torsion and curvature of the reference metrics} \label{sectorsion}

This section is devoted to proving a technical lemma on estimates for the torsion and curvature of the reference metrics.  These estimates will be needed later in Sections \ref{sectiontrace} and \ref{sectioncalabi}.

Recall that the reference forms $\tilde{\omega}=\tilde{\omega}(t)$ are given by
$$\tilde{\omega} = e^{-t} \of + (1-e^{-t}) \omega_S.$$
For $t \ge T_I$, this defines a Hermitian metric which we denote by  $\tilde{g}$.
We will use a tilde to denote quantities with respect to $\tilde{g}$, such as $\tilde{T}_{ij}^k$ for  the torsion tensor,  $\tilde{\nabla}$ for the Chern connection and $\widetilde{\textrm{Rm}}$ for the Chern curvature tensor.  We will write $\db\ti{T}$ for the tensor $\de_{\ov{\ell}}\ti{T}^i_{jk}=\ti{\nabla}_{\ov{\ell}}
\ti{T}^i_{jk}$.

Denote by $(T^0)_{ij}^k$ the torsion tensor of the initial metric $g_0$, and
 $T^0_{ij\ov{\ell}}=(T^0)^k_{ij} (g_0)_{k\ov{\ell}}.$ Then
since $d\ti{\omega}=e^{-t}d\omega_0$, we have $\ti{T}_{ij\ov{\ell}}=e^{-t}T^0_{ij\ov{\ell}}.$

\begin{lemma} \label{lemmatildegestimates}  There exists a uniform constant $C$ such that for $t \ge T_I$,
\begin{enumerate}
\item[(i)]  $\displaystyle{ | \tilde{T}|_{\tilde{g}} \le C}$.
\smallskip
\item[(ii)]  $\displaystyle{| \db \tilde{T}|_{\tilde{g}} + | \ti{\nabla} \tilde{T}|_{\tilde{g}} +  | \widetilde{\emph{Rm}} |_{\tilde{g}}  \le C e^{t/2}}$.
\smallskip
\item[(iii)] $\displaystyle{|\ov{\ti{\nabla}}\db\ti{T}|_g+ |\ti{\nabla}\db\ti{T}|_g\leq Ce^{t}}$.
\end{enumerate}
\end{lemma}
\begin{proof}
We may choose local product holomorphic coordinates  $z^1, z^2$, independent of $t$, with $z^1$ in the fiber direction and $z^2$ in the base direction.  Since $\gf$ is flat in the $z^1$ direction, we may assume that derivatives of $(\gf)_{1\ov{1}}$ in the $z^1$ direction vanish.
Now with respect to these coordinates,
we may write
\begin{align}  \label{localgt1}
\tilde{g}_{1\ov{1}} ={} & e^{-t} (g_{\textrm{flat}})_{1\ov{1}}, \quad \tilde{g}_{1\ov{2}} = e^{-t} (g_{\textrm{flat}})_{1\ov{2}}\\ \label{localgt2}
\tilde{g}_{2\ov{1}} ={} & e^{-t} (g_{\textrm{flat}})_{2\ov{1}}, \quad \tilde{g}_{2\ov{2}} = e^{-t}(g_{\textrm{flat}})_{2\ov{2}} + (1-e^{-t})(g_S)_{2\ov{2}},
\end{align}
where we are writing $g_S$ for $\pi^* g_S$.
Then a straightforward computation shows that there exists a uniform constant $C>0$ so that
$$\frac{e^{-t}}{C} \le \tilde{g}_{1\ov{1}} \le C e^{-t}, \quad \frac{e^{t}}{C} \le \tilde{g}^{1\ov{1}} \le C e^{t}$$
$$ \frac{1}{C} \le \tilde{g}_{2\ov{2}} \le C, \quad \frac{1}{C} \le \tilde{g}^{2\ov{2}} \le C$$
$$| \tilde{g}_{1\ov{2}}| \le Ce^{-t} \quad | \tilde{g}^{1\ov{2}}| \le C.$$
Then
$$| \tilde{T}|_{\tilde{g}}^2 = \tilde{g}^{i\ov{j}} \tilde{g}^{k\ov{\ell}} \tilde{g}^{p\ov{q}} \tilde{T}_{ik\ov{q}} \ov{ \tilde{T}_{j \ell \ov{p}}} = e^{-2t}\tilde{g}^{i\ov{j}} \tilde{g}^{k\ov{\ell}} \tilde{g}^{p\ov{q}} \Tf_{ik\ov{q}} \ov{ \Tf_{j \ell \ov{p}}}.$$
This is uniformly bounded since the only unbounded terms of type $\tilde{g}^{i\ov{j}}$ are the terms $\tilde{g}^{1\ov{1}}$, but by the skew-symmetry of $\Tf_{ij\ov{\ell}}$ in $i,j$, there can be at most two such terms in the above expression, and each is bounded above by $Ce^t$ (since the components  $\Tf_{ij\ov{k}}$ are all uniformly bounded in the holomorphic coordinates $z^1, z^2$).  This completes the proof of (i).

For (ii), note that we may choose the coordinates $z^1, z^2$ as above with the additional property that at a fixed point  $x$ say, the derivative $\partial_{2} g_S$ vanishes.  This implies that at $x$ we have $\partial_i \tilde{g}_{j\ov{\ell}} = e^{-t} \partial_i (\gf)_{j\ov{\ell}}$ for all $i,j,\ell$. Note that since our coordinates are independent of $t$ and depend continuously on the point $x\in M$, we may allow our constants to depend on this choice of coordinate system.

We first claim that
at $x$,
\begin{equation} \label{Gammabd}
| \tilde{\Gamma}_{ik}^p |^2_{\tilde{g}}:= \tilde{g}^{i\ov{j}} \tilde{g}^{k\ov{\ell}} \tilde{g}_{p\ov{q}} \tilde{\Gamma}_{ik}^p \ov{\tilde{\Gamma}_{j\ell}^q} \le C e^{t},
\end{equation}
where $\tilde{\Gamma}^p_{ik}$ are the Christoffel symbols of the Chern connection of $\tilde{g}$.
Note that since $\tilde{\Gamma}^p_{ik}$ is not a tensor, this quantity depends on our choice of coordinates.

For \eqref{Gammabd} compute,
$$| \tilde{\Gamma}^p_{ik} |^2_{\tilde{g}}  = \tilde{g}^{i\ov{j}} \tilde{g}^{k\ov{\ell}} \tilde{g}^{p\ov{q}} \partial_i \tilde{g}_{k\ov{q}} \partial_{\ov{j}} \tilde{g}_{p\ov{\ell}}= e^{-2t} \tilde{g}^{i\ov{j}} \tilde{g}^{k\ov{\ell}} \tilde{g}^{p\ov{q}} \partial_i (\gf)_{k\ov{q}} \partial_{\ov{j}} (\gf)_{p\ov{\ell}}.$$
But this is bounded from above by $Ce^{t}$ since each term of type $\tilde{g}^{i\ov{j}}$ is bounded from above by $Ce^t$.

Next note that, at $x$,
\begin{equation} \label{barT}
| \partial_{\ov{\ell}} \Tf_{ij\ov{k}}|_{\tilde{g}}^2 \le C e^{3t}.
\end{equation}
Indeed this follows from the skew symmetry of $\partial_{\ov{\ell}} \Tf_{ij\ov{k}}$ (again, not a tensor) in the indices $i$ and $j$. Then at $x$,
\begin{equation}\begin{split}\label{ii1}
| \db \tilde{T} |^2_{\tilde{g}} = & e^{-2t} | \tilde{\nabla}_{\ov{\ell}} \Tf_{ij\ov{k}}|^2_{\tilde{g}} \\
= & e^{-2t} | \partial_{\ov{\ell}} \Tf_{ij\ov{k}} - \ov{\tilde{\Gamma}^r_{\ell k}} \Tf_{ij\ov{r}}|^2_{\tilde{g}} \\
\le & 2 e^{-2t} | \partial_{\ov{\ell}} \Tf_{ij\ov{k}}|_{\tilde{g}}^2 + 2e^{-2t}| \tilde{\Gamma}^r_{\ell k}|^2_{\tilde{g}} | \Tf_{ij\ov{r}}|^2_{\tilde{g}} \le C e^t
\end{split}\end{equation}
where the last inequality follows from
 \eqref{Gammabd}, \eqref{barT} and the fact that $| \Tf_{ij\ov{r}}|^2_{\tilde{g}} \le Ce^{2t}$.

The bound on $|\tilde{\nabla} \tilde{T}|_{\tilde{g}}$ is completely analogous (again we compute at $x$):
\begin{equation}\begin{split}\label{ii2}
| \ti{\nabla} \tilde{T} |^2_{\tilde{g}} = & e^{-2t} | \tilde{\nabla}_{\ell} \Tf_{ij\ov{k}}|^2_{\tilde{g}} \\
=  & e^{-2t} | \partial_{\ell} \Tf_{ij\ov{k}} - \tilde{\Gamma}^r_{\ell i} \Tf_{rj\ov{k}}- \tilde{\Gamma}^r_{\ell j} \Tf_{ir\ov{k}}|^2_{\tilde{g}} \\
\le & 2 e^{-2t} | \partial_{\ell} \Tf_{ij\ov{k}}|_{\tilde{g}}^2 + 4e^{-2t}| \tilde{\Gamma}^r_{\ell k}|^2_{\tilde{g}} | \Tf_{ij\ov{k}}|^2_{\tilde{g}}\leq Ce^t.
\end{split}\end{equation}

For the bound on the curvature $\tilde{R}_{i\ov{j}k\ov{\ell}}$ of $\tilde{g}$,  we first compute in our coordinates,
\begin{equation} \label{1111}
| \tilde{R}_{1\ov{1} 1\ov{1}}| \le Ce^{-2t},
\end{equation}
where by $| \cdot |$ we mean the absolute value as a complex (or real) number.  Recall that the Chern curvature of $\tilde{g}$ is given by
$$\ti{R}_{i\ov{j}k \ov{\ell}} = -\partial_i \partial_{\ov{j}} \tilde{g}_{k\ov{\ell}} + \tilde{g}^{p\ov{q}} \partial_{i} \tilde{g}_{k\ov{q}} \partial_{\ov{j}} \tilde{g}_{p\ov{\ell}}.$$
Hence from \eqref{localgt1} and \eqref{localgt2} and the fact that $\partial_1\partial_{\ov{1}} (\gf)_{1\ov{1}}$ and $\partial_1 (\gf)_{1\ov{1}}$ vanish in our coordinate system,
\begin{align*}
\tilde{R}_{1\ov{1}1\ov{1}} ={} & -\partial_1 \partial_{\ov{1}} \tilde{g}_{1\ov{1}} + \tilde{g}^{p\ov{q}} \partial_{1} \tilde{g}_{1\ov{q}} \partial_{\ov{1}} \tilde{g}_{p\ov{1}} \\
= {} &  e^{-2t} \sum_{(p,q)\neq (1,1)} \tilde{g}^{p\ov{q}} \partial_{1} (g_{\textrm{flat}})_{1\ov{q}} \partial_{\ov{1}} (g_{\textrm{flat}})_{p\ov{1}}.
\end{align*}
but since each term $\tilde{g}^{p\ov{q}}$ for $(p,q) \neq (1,1)$ is uniformly bounded, this gives \eqref{1111}.

Next we show that
\begin{equation} \label{2222}
| \tilde{R}_{2\ov{2} 2\ov{2}} | \le C.
\end{equation}
For this note that  $|\partial_2 \partial_{\ov{2}} \tilde{g}_{2\ov{2}}|\le C$ and
\begin{align*}
\sum_{p,q} \tilde{g}^{p\ov{q}} \partial_2 \tilde{g}_{2\ov{q}} \partial_{\ov{2}} \tilde{g}_{p\ov{2}} = {}&  \tilde{g}^{1\ov{1}} \partial_2 \tilde{g}_{2\ov{1}} \partial_{\ov{2}} \tilde{g}_{1\ov{2}} + \sum_{(p,q) \neq (1,1)} \tilde{g}^{p\ov{q}} \partial_2 \tilde{g}_{2\ov{q}} \partial_{\ov{2}} \tilde{g}_{p\ov{2}},
\end{align*}
but the first term is of order $O(e^{-t})$ and the second is uniformly bounded.  This proves \eqref{2222}.

Finally, we show that
\begin{equation} \label{ijkl}
| \tilde{R}_{i\ov{j} k\ov{\ell}} | \le Ce^{-t}, \quad \textrm{for } (i,j,k, \ell) \ \textrm{not all equal}.
\end{equation}
To see this observe that, for $i,j,k, \ell$ not all equal, we have
\begin{equation} \label{ijkl1}
| \partial_i \partial_{\ov{j}} \ti{g}_{k\ov{\ell}} | \le Ce^{-t}.
\end{equation}
Indeed, this follows immediately from \eqref{localgt1} and \eqref{localgt2} unless $k=\ell=2$.  But then one of $i$ or $j$ must equal 1 and we use the fact that $\partial_{1} (g_S)_{2\ov{2}}=0$.

Moreover, for $i,j,k, \ell$ not all equal,  we claim:
\begin{equation} \label{ijkl2}
\left| \sum_{p,q} \tilde{g}^{p\ov{q}} \partial_i \tilde{g}_{k\ov{q}} \partial_{\ov{j}} \tilde{g}_{p\ov{\ell}} \right| \le Ce^{-t}.
\end{equation}
Indeed, first assume that neither $k$ nor $\ell$ is equal to 2.  Then
$$| \tilde{g}^{p\ov{q}}|\le C e^t, \ | \partial_i \tilde{g}_{k\ov{q}}| \le Ce^{-t} \quad \textrm{and} \quad |\partial_{\ov{j}} \tilde{g}_{p\ov{\ell}}| \le Ce^{-t},$$
and so
$$| \tilde{g}^{p\ov{q}} \partial_i \tilde{g}_{k\ov{q}} \partial_{\ov{j}} \tilde{g}_{p\ov{\ell}}| \le Ce^{-t}.$$
Next suppose that $k=2$ and $\ell=1$.  Then
\begin{align*}
\sum_{p,q} \tilde{g}^{p\ov{q}} \partial_i \tilde{g}_{2\ov{q}} \partial_{\ov{j}} \tilde{g}_{p\ov{1}}
={} & \sum_{p} \tilde{g}^{p\ov{1}} \partial_i \tilde{g}_{2\ov{1}} \partial_{\ov{j}} \tilde{g}_{p\ov{1}} + \sum_{p} \tilde{g}^{p\ov{2}} \partial_i \tilde{g}_{2\ov{2}} \partial_{\ov{j}} \tilde{g}_{p\ov{1}}.
\end{align*}
The  first term on the right hand side is of order $O(e^{-t})$ by the same argument as above, and the second is of order $O(e^{-t})$ since $|\tilde{g}^{p\ov{2}}| \le C$, $|\partial_i \tilde{g}_{2\ov{2}}| \le C$ and $| \partial_{\ov{j}} \tilde{g}_{p\ov{1}}| \le Ce^{-t}$.  This proves \eqref{ijkl2} if $k=2$ and $\ell=1$ and the case $k=1$, $\ell=2$ is similar.

Finally we deal with the case $k=\ell=2$.   We wish to bound
\begin{align*}
\sum_{p,q} \tilde{g}^{p\ov{q}} \partial_i \tilde{g}_{2\ov{q}} \partial_{\ov{j}} \tilde{g}_{p\ov{2}}.
\end{align*}
If one of $p$ or $q$ is equal to 1 then the summand is of order $O(e^{-t})$ by similar arguments to the ones given above.  Otherwise the summand is
$$\tilde{g}^{2\ov{2}} \partial_i \tilde{g}_{2\ov{2}} \partial_{\ov{j}} \tilde{g}_{2\ov{2}},$$
and then we use that fact that one of $i$ or $j$ must be 1, since we are assuming that $i,j,k,\ell$ are not all equal.  But $|\partial_1 \tilde{g}_{2\ov{2}}| \le C e^{-t}$ and so the summand is of order $O(e^{-t})$.  This completes the proof of \eqref{ijkl2}.

Combining \eqref{ijkl1} and \eqref{ijkl2} gives \eqref{ijkl}.

To complete the proof of (ii), we note that
\begin{align*}
| \widetilde{\textrm{Rm}}|^2_{\tilde{g}} = \tilde{g}^{i\ov{q}} \tilde{g}^{p \ov{j}} \tilde{g}^{k \ov{s}} \tilde{g}^{r \ov{\ell}} \tilde{R}_{i\ov{j} k\ov{\ell}} \ov{ \tilde{R}_{q\ov{p} s\ov{r}}}.
\end{align*}
Recall that $\tilde{g}^{a\ov{b}}$ is bounded by $C$ if $(a,b)\neq (1,1)$ and by $Ce^t$ if $a=b=1$.
When $i=j=k=\ell=p=q=r=s=1$, then  we apply \eqref{1111} to see that the summand is bounded by $C$.  We get the same bound if $i=j=k=\ell=2$ or $p=q=r=s=2$ by applying \eqref{1111}, \eqref{2222} and \eqref{ijkl}.
Otherwise, both of the terms $\tilde{R}_{i\ov{j} k\ov{\ell}}$ and $\tilde{R}_{q\ov{p} s\ov{r}}$ are bounded by $Ce^{-t}$ (or better) by \eqref{1111} and \eqref{ijkl} (because
the term $\tilde{R}_{2\ov{2}2\ov{2}}$ does not appear).  Moreover, at least one of the metric terms $\tilde{g}^{-1}$ is bounded by a uniform constant $C$, while the other three are each bounded from above by $Ce^t$.  Thus in every case we obtain $| \widetilde{\textrm{Rm}}|^2_{\tilde{g}} \le Ce^{t}$.  Combining this with (\ref{ii1}) and (\ref{ii2}) gives (ii).

For (iii), compute
\[\begin{split}
\ti{\nabla}_{\ov{p}} \ti{\nabla}_{\ov{q}}\ti{T}_{ij\ov{k}}
=&\ti{\nabla}_{\ov{p}}(\de_{\ov{q}}\ti{T}_{ij\ov{k}} -\ov{\ti{\Gamma}^\ell_{qk}}\ti{T}_{ij\ov{\ell}} )\\
=&\de_{\ov{p}}\de_{\ov{q}}\ti{T}_{ij\ov{k}}-\ti{T}_{ij\ov{\ell}}\de_{\ov{p}}\ov{\ti{\Gamma}^\ell_{qk}}   -\ov{\ti{\Gamma}^\ell_{qk}}\de_{\ov{p}}\ti{T}_{ij\ov{\ell}}
-\ov{\ti{\Gamma}^\ell_{pq}}\de_{\ov{\ell}}\ti{T}_{ij\ov{k}}-\ov{\ti{\Gamma}^\ell_{pk}}\de_{\ov{q}}\ti{T}_{ij\ov{\ell}}\\
&+\ov{\ti{\Gamma}^r_{pq}}\ov{\ti{\Gamma}^\ell_{rk}}\ti{T}_{ij\ov{\ell}}+\ov{\ti{\Gamma}^r_{pk}}\ov{\ti{\Gamma}^\ell_{qr}}\ti{T}_{ij\ov{\ell}}.
\end{split}\]
For the first term we observe that
\begin{equation} \label{doubledbarT}
|\de_{\ov{p}}\de_{\ov{q}}\ti{T}_{ij\ov{k}}|^2_{\ti{g}}=e^{-2t}|\de_{\ov{p}}\de_{\ov{q}}T^0_{ij\ov{k}}|^2_{\ti{g}}\leq Ce^{2t},
\end{equation}
because of the skew symmetry in $i,j$.  And, as in the proof of (ii),
\begin{equation} \label{singledbarT}
| \partial_{\ov{p}} \tilde{T}_{ij\ov{\ell}}|^2_{\tilde{g}} = e^{-2t} | \partial_{\ov{p}} T^0_{ij\ov{k}}|^2_{\tilde{g}} \le Ce^t.
\end{equation}
We claim that
\begin{equation} \label{claimGamma}
|\de_p \ti{\Gamma}^i_{jk}|^2_{\ti{g}}:=\tilde{g}^{p\ov{q}} \tilde{g}^{j\ov{a}} \tilde{g}^{k \ov{b}} \tilde{g}_{i\ov{c}} \partial_p \tilde{\Gamma}^i_{jk} \ov{ \partial_{q} \tilde{\Gamma}^c_{ab}} \le C e^{2t}.
\end{equation}
To see this note that at $x$ we have
\begin{align*}
\de_p \ti{\Gamma}^i_{jk}={} & \ti{g}^{i\ov{\ell}}\de_p\de_j\ti{g}_{k\ov{\ell}}-\ti{g}^{i\ov{s}}\ti{g}^{r\ov{\ell}}\de_p\ti{g}_{r\ov{s}}
\de_j\ti{g}_{k\ov{\ell}} \\
={} & e^{-t} \ti{g}^{i\ov{\ell}}\de_p\de_j (\gf)_{k\ov{\ell}}+(1-e^{-t}) \ti{g}^{i\ov{\ell}}\de_p\de_j (g_S)_{k\ov{\ell}}\\
&- e^{-2t} \ti{g}^{i\ov{s}}\ti{g}^{r\ov{\ell}}\de_p(\gf)_{r\ov{s}}
\de_j (\gf)_{k\ov{\ell}},
\end{align*}
and so (with the obvious notation)
$$|\de_p \ti{\Gamma}^i_{jk}|^2_{\ti{g}}\leq C\left( e^{-2t}|\de_p\de_j (\gf)_{k\ov{\ell}}|^2_{\ti{g}}+|\de_p\de_j (g_S)_{k\ov{\ell}}|^2_{\ti{g}}+e^{-4t}|\de_p (\gf)_{r\ov{s}}
\de_j (\gf)_{k\ov{\ell}}|^2_{\ti{g}}\right).$$
But the second term equals $C|\de_2\de_2 (g_S)_{2\ov{2}}|^2 |\ti{g}^{2\ov{2}}|^4,$ and so is bounded by $C$,
while the other two terms are bounded by $Ce^{2t}$ because each term of type $\ti{g}^{i\ov{j}}$ is bounded above by $Ce^t$. This establishes the claim \eqref{claimGamma}.

Combining \eqref{Gammabd}, \eqref{doubledbarT}, \eqref{singledbarT}, \eqref{claimGamma} and parts (i) and (ii) we have proved the bound $|\ov{\ti{\nabla}}\db\ti{T}|_g\le Ce^t$.

Finally, calculate
\[\begin{split}
\ti{\nabla}_{p} \ti{\nabla}_{\ov{q}}\ti{T}_{ij\ov{k}}
=&\ti{\nabla}_{p}(\de_{\ov{q}}\ti{T}_{ij\ov{k}} -\ov{\ti{\Gamma}^\ell_{qk}}\ti{T}_{ij\ov{\ell}} )\\
=&\de_{p}\de_{\ov{q}}\ti{T}_{ij\ov{k}}-\ti{T}_{ij\ov{\ell}}\de_{p}\ov{\ti{\Gamma}^\ell_{qk}}   -\ov{\ti{\Gamma}^\ell_{qk}}\de_{p}\ti{T}_{ij\ov{\ell}}
-\ti{\Gamma}^\ell_{pi}\de_{\ov{q}}\ti{T}_{\ell j\ov{k}}-\ti{\Gamma}^\ell_{pj}\de_{\ov{q}}\ti{T}_{i\ell\ov{k}}\\
&+\ti{\Gamma}^r_{pi}\ov{\ti{\Gamma}^\ell_{qk}}\ti{T}_{r j\ov{\ell}}+\ti{\Gamma}^r_{pj}\ov{\ti{\Gamma}^\ell_{qk}}\ti{T}_{ir\ov{\ell}}.
\end{split}\]
Arguing as above we can bound the $|\cdot |_{\ti{g}}$ of all these terms by $Ce^t$.
Indeed, the only  terms which are different are the one involving $\de_{p}\ov{\ti{\Gamma}^\ell_{qk}}=-\ov{\ti{R}_{q\ov{p}k}{}^{\ell}}$,
which can be bounded using part (ii) and the one involving $\de_{p}\ti{T}_{ij\ov{\ell}}$ which can be bounded by the same argument as in \eqref{singledbarT}. This finishes the proof of (iii).
\end{proof}

\vskip 2\baselineskip

\section{Evolution of the trace of the metric} \label{sectiontrace}

Let $\omega=\omega(t)$ solve the normalized Chern-Ricci flow (\ref{NCRF}) in the setting of Theorem \ref{main}. The main theorem we prove in this section is:

\begin{theorem}   \label{metricbdthm}
There exists a uniform constant $C>0$ such that for $t \ge T_I$,
$$\emph{\textrm{tr}}_{\tilde{\omega}} \omega \le C.$$
Hence the metrics $\omega$ and $\tilde{\omega}$ are uniformly equivalent for $t \ge T_I$.
\end{theorem}

For the last assertion, note that we have
\begin{equation}
\tr{\omega}{\tilde{\omega}} = \frac{\tilde{\omega}^2}{\omega^2} \tr{\tilde{\omega}}{\omega},
\end{equation}
and the uniform equivalence of the volume forms $\omega^2$ and $\tilde{\omega}^2$ (Lemma \ref{lowerbound}).  Then
an upper bound for $\tr{\tilde{\omega}}{\omega}$ is equivalent to an upper bound for $\tr{\omega}{\tilde{\omega}}$ and hence also the uniform equivalence of $\omega$ and $\tilde{\omega}$.

In \cite{SW} a similar estimate is proved using a direct maximum principle argument and a bound for the potential $\varphi$.  Here, as discussed in the Introduction, there are new unbounded terms arising from the torsion.  We will control these terms using the exponential decay estimate for $\varphi$ (Lemma \ref{key}).

First we need the following lemma.

\begin{lemma} \label{lemmalogtr} For $t \ge T_I$, the following evolution inequality holds:
\begin{equation} \label{x}
\left( \ddt{} - \Delta \right) \log \tr{\tilde{\omega}}{\omega} \le \frac{2}{(\tr{\tilde{\omega}}{\omega})^2} \emph{Re} \left( \tilde{g}^{i\ov{\ell}} g^{k\ov{q}} \tilde{T}_{k i \ov{\ell}} \partial_{\ov{q}} \tr{\tilde{\omega}}{\omega} \right) + Ce^{t/2}  \tr{\omega}{\tilde{\omega}}.
\end{equation}
\end{lemma}
\begin{proof}  From \cite[Proposition 3.1]{TW} we have
\begin{align*}
\lefteqn{\left( \ddt{} - \Delta \right) \log \tr{\ti{\omega}}{\omega}} \\= {} & \frac{1}{\tr{\ti{\omega}}{\omega}} \left( -g^{p\ov{j}}g^{i\ov{q}}\ti{g}^{k\ov{\ell}}\ti{\nabla}_k g_{i\ov{j}}\ti{\nabla}_{\ov{\ell}}g_{p\ov{q}}
+ \frac{1}{\tr{\tilde{\omega}}{\omega}} g^{k\ov{\ell}} \de_k \tr{\tilde{\omega}}{\omega}  \de_{\ov{\ell}} \tr{\tilde{\omega}}{\omega} \right.  \\
& - 2\mathrm{Re}\left(g^{i\ov{j}}\ti{g}^{k\ov{\ell}} \ti{T}^p_{ki}\ti{\nabla}_{\ov{\ell}}g_{p\ov{j}}\right) -g^{i\ov{j}}\ti{g}^{k\ov{\ell}}\ti{T}^p_{ik}\ov{\ti{T}^q_{j\ell}}g_{p\ov{q}}  \\
&+g^{i\ov{j}}\ti{g}^{k\ov{\ell}}(\ti{\nabla}_i\ov{\ti{T}^q_{j\ell}}-\ti{R}_{i\ov{\ell}p\ov{j}}\ti{g}^{p\ov{q}})g_{k\ov{q}}
-g^{i\ov{j}}\ti{\nabla}_i   \ov{\ti{T}^\ell_{j\ell}} -g^{i\ov{j}}\ti{g}^{k\ov{\ell}}\ti{g}_{p \ov{j}}\ti{\nabla}_{\ov{\ell}}  \ti{T}^p_{ik}  \\
&\left.
  +g^{i\ov{j}}\ti{g}^{k\ov{\ell}} \ti{T}^p_{ik} \ov{\ti{T}^q_{j\ell}}\ti{g}_{p \ov{q}}
 -\tr{\ti{\omega}}{\omega}
 - \ti{g}^{i\ov{\ell}}\ti{g}^{k\ov{j}}g_{i\ov{j}} \frac{\partial}{\partial t} \tilde{g}_{k\ov{\ell}} \right).
\end{align*}
Note that there are some differences from the computation in \cite{TW} since here we are evolving $\omega$ by the \emph{normalized} Chern-Ricci flow, and our reference metrics $\tilde{\omega}$ depend on time.  In particular here we have $T_{ij\ov{k}}= \tilde{T}_{ij\ov{k}}$ (instead of $T_{ij\ov{k}} = (T_0)_{ij\ov{k}}$ in \cite{TW}).  Also, the last two terms above are new:  the first arising from the $-\omega$ term on the right hand side of (\ref{NCRF}) and the second from the time derivative of $\tilde{\omega}$.
Fortunately, the contribution of these two terms has a good sign.  Indeed,
 observe that $\ddt{} \tilde{g} = g_S -\tilde{g} \ge -\tilde{g}$ and hence
$$- \tr{\tilde{\omega}}{\omega} - \ti{g}^{i\ov{\ell}}\ti{g}^{k\ov{j}}g_{i\ov{j}} \frac{\partial}{\partial t} \tilde{g}_{k\ov{\ell}}  \le 0.$$
Again from Proposition 3.1 in \cite{TW}, we have
\begin{align*}
& \frac{1}{\tr{\ti{\omega}}{\omega}} \bigg( -g^{p\ov{j}}g^{i\ov{q}}\ti{g}^{k\ov{\ell}}\ti{\nabla}_k g_{i\ov{j}}\ti{\nabla}_{\ov{\ell}}g_{p\ov{q}}
+ \frac{1}{\tr{\tilde{\omega}}{\omega}} g^{k\ov{\ell}} \de_k \tr{\tilde{\omega}}{\omega}  \de_{\ov{\ell}} \tr{\tilde{\omega}}{\omega}  \\
&  - 2\mathrm{Re}\left(g^{i\ov{j}}\ti{g}^{k\ov{\ell}} \ti{T}^p_{ki}\ti{\nabla}_{\ov{\ell}}g_{p\ov{j}}\right) -g^{i\ov{j}}\ti{g}^{k\ov{\ell}}\ti{T}^p_{ik}\ov{\ti{T}^q_{j\ell}}g_{p\ov{q}} \bigg)  \le  \frac{2}{(\tr{\ti{\omega}}{\omega})^2} \textrm{Re} \left( \ti{g}^{i\ov{\ell}} g^{k\ov{q}} \tilde{T}_{k i \ov{\ell}} \partial_{\ov{q}} \tr{\ti{\omega}}{\omega} \right).
\end{align*}
Hence to complete the proof of the lemma it remains to show that for $t\geq T_I$ we have
\begin{align*}
& \frac{1}{\tr{\ti{\omega}}{\omega}}  \left( g^{i\ov{j}}\ti{g}^{k\ov{\ell}}(\ti{\nabla}_i\ov{\ti{T}^q_{j\ell}}-\ti{R}_{i\ov{\ell}p\ov{j}}\ti{g}^{p\ov{q}})g_{k\ov{q}}
-g^{i\ov{j}}\ti{\nabla}_i   \ov{\ti{T}^\ell_{j\ell}}   -
 g^{i\ov{j}}\ti{g}^{k\ov{\ell}}\ti{g}_{p \ov{j}}\ti{\nabla}_{\ov{\ell}}  \ti{T}^p_{ik}  \right.  \\
 & \left. +g^{i\ov{j}}\ti{g}^{k\ov{\ell}} \ti{T}^p_{ik} \ov{\ti{T}^q_{j\ell}}\ti{g}_{p \ov{q}} \right) \le C(\tr{\omega}{\tilde{\omega}}) e^{t/2}.
\end{align*}
But this follows easily from Lemma \ref{lemmatildegestimates},  the fact that the quantities $\tr{\omega}{\tilde{\omega}}$ and $\tr{\tilde{\omega}}{\omega}$ are uniformly equivalent and the inequality $\tr{\tilde{\omega}}{\omega} \ge C^{-1}>0$ for a uniform constant $C$ (the geometric-arithmetic means inequality). Indeed,
$$\frac{1}{\tr{\ti{\omega}}{\omega}} | g^{i\ov{j}}\ti{g}^{k\ov{\ell}}\ti{\nabla}_i\ov{\ti{T}^q_{j\ell}} g_{k\ov{q}} |  \le \frac{1}{\tr{\ti{\omega}}{\omega}} | g^{-1} |_{\tilde{g}} | \tilde{g}^{-1}|_{\tilde{g}} | \ov{\tilde{\nabla}} \tilde{T}|_{\tilde{g}} |g|_{\tilde{g}} \le  C(\tr{\omega}{\tilde{\omega}}) e^{t/2},$$
$$\frac{1}{\tr{\ti{\omega}}{\omega}} |g^{i\ov{j}}\ti{g}^{k\ov{\ell}}\ti{g}^{p\ov{q}}g_{k\ov{q}}\ti{R}_{i\ov{\ell}p\ov{j}}|
\leq \frac{1}{\tr{\ti{\omega}}{\omega}} | g^{-1} |_{\tilde{g}}| \tilde{g}^{-1}|^2_{\tilde{g}} |g|_{\tilde{g}} | \widetilde{\textrm{Rm}}|_{\tilde{g}} \leq  C(\tr{\omega}{\tilde{\omega}}) e^{t/2},$$
$$\frac{1}{\tr{\ti{\omega}}{\omega}} |g^{i\ov{j}}\ti{\nabla}_i   \ov{\ti{T}^\ell_{j\ell}} |\leq \frac{1}{\tr{\ti{\omega}}{\omega}} | g^{-1} |_{\tilde{g}} | \ov{\tilde{\nabla}} \tilde{T}|_{\tilde{g}}\leq C e^{t/2},$$
$$\frac{1}{\tr{\ti{\omega}}{\omega}} | g^{i\ov{j}}\ti{g}^{k\ov{\ell}}\ti{g}_{p \ov{j}}\ti{\nabla}_{\ov{\ell}}  \ti{T}^p_{ik}  |\leq
\frac{1}{\tr{\ti{\omega}}{\omega}} | g^{-1} |_{\tilde{g}} |\ti{g}^{-1}|_{\ti{g}} |\ti{g}|_{\ti{g}} |\ov{\tilde{\nabla}} \tilde{T}|_{\tilde{g}}\leq C e^{t/2},$$
$$\frac{1}{\tr{\ti{\omega}}{\omega}} | g^{i\ov{j}}\ti{g}^{k\ov{\ell}} \ti{T}^p_{ik}\ov{\ti{T}^q_{j\ell}} \ti{g}_{p \ov{q}}|\leq \frac{1}{\tr{\ti{\omega}}{\omega}}
| g^{-1} |_{\tilde{g}} |\ti{g}^{-1}|_{\ti{g}} |\ti{g}|_{\ti{g}} |\ti{T}|^2_{\ti{g}}\leq C,$$
as required.
\end{proof}

We can now prove Theorem \ref{metricbdthm}, making use of the decay estimate of $\varphi$ (Lemma \ref{key}) and the bound on $\dot{\varphi}$ (Lemma \ref{lowerbound}).

\begin{proof}[Proof of Theorem \ref{metricbdthm}]
We use the fact that $e^{t/2} \varphi$ is uniformly bounded, and consider the quantity
$$Q = \log \tr{\tilde{\omega}}{\omega} - A e^{t/2} \varphi + \frac{1}{\tilde{C} + e^{t/2} \varphi},$$
where $\tilde{C}$ is a uniform constant chosen so that $\tilde{C}+e^{t/2} \varphi \ge 1$, and  $A$ is a large constant to be determined later.   The idea of adding an extra term, of the form of a  reciprocal of a potential function, comes from
 Phong-Sturm \cite{PS} and was
used in the context of the Chern-Ricci flow in \cite{TW}. Notice  that
$$0 \le \frac{1}{\tilde{C} + e^{t/2} \varphi} \le 1.$$
We will show that at a point $(x_0, t_0)$ with $t_0 > T_I$ at which $Q$ achieves a maximum, we have a uniform upper bound of $\tr{\tilde{\omega}}{\omega}$, and the theorem will follow
thanks to Lemma \ref{key}.

First compute, using the fact that $\Delta \varphi = 2 - \tr{\omega}{\tilde{\omega}}$ and the bounds for $\vp$ and $\dot{\varphi}$ from Lemma \ref{lowerbound},
\begin{equation}\label{eqnet2} \begin{split}
 \left( \ddt{} - \Delta \right) &\left( - A e^{t/2} \varphi + \frac{1}{\tilde{C} + e^{t/2} \varphi}\right)    \\
 =&- \left( A + \frac{1}{(\tilde{C}+ e^{t/2} \varphi)^2} \right)  \left(e^{t/2} \dot{\varphi} + \frac{1}{2}e^{t/2} \varphi \right) \\
 &+ \left( A + \frac{1}{(\tilde{C}+ e^{t/2} \varphi)^2} \right) \Delta ( e^{t/2} \varphi) - \frac{2 | \partial (e^{t/2} \varphi) |_{g}^2}{(\tilde{C} + e^{t/2} \varphi)^3} \\
 \le & CA e^{t/2} - A e^{t/2} \tr{\omega}{\tilde{\omega}} - \frac{2 | \partial (e^{t/2} \varphi) |_{g}^2}{(\tilde{C} + e^{t/2} \varphi)^3}.
\end{split}\end{equation}

At the point $(x_0, t_0)$, we have $\partial_{\ov{q}} Q = 0$, which implies that
$$\frac{\partial_{\ov{q}} \tr{\tilde{\omega}}{\omega}}{\tr{\tilde{\omega}}{\omega}} = \left( A + \frac{1}{(\tilde{C} + e^{t/2} \varphi)^2} \right)  e^{t/2} \partial_{\ov{q}} \varphi.$$
Then at this point,
\begin{equation}\label{dagger} \begin{split}
\frac{2}{(\tr{\tilde{\omega}}{\omega})^2} &\textrm{Re} \left( \tilde{g}^{i\ov{\ell}} g^{k\ov{q}} \tilde{T}_{k i \ov{\ell}} \partial_{\ov{q}} \tr{\tilde{\omega}}{\omega} \right) \\
& = \frac{2}{\tr{\tilde{\omega}}{\omega}} \textrm{Re} \left( \tilde{g}^{i\ov{\ell}} g^{k\ov{q}} \tilde{T}_{ki\ov{\ell}} \left( A + \frac{1}{(\tilde{C} + e^{t/2} \varphi)^2} \right)  e^{t/2} \partial_{\ov{q}} \varphi \right) \\
& \le \frac{CA^2}{(\tr{\tilde{\omega}}{\omega})^2} (\tilde{C} + e^{t/2} \varphi)^3  g^{k\ov{q}} \tilde{g}^{i\ov{\ell}}  \tilde{T}_{k i \ov{\ell}} \tilde{g}^{m \ov{j}} \ov{\tilde{T}_{q j \ov{m}}} + \frac{ | \partial(e^{t/2} \varphi)|^2_g}{(\tilde{C}+ e^{t/2} \varphi)^3} \\
& \le \frac{CA^2}{\tr{\tilde{\omega}}{\omega}}   + \frac{ | \partial(e^{t/2} \varphi)|^2_g}{(\tilde{C}+ e^{t/2} \varphi)^3},
\end{split}\end{equation}
where for the last step we have used Lemma \ref{key}, part (i) of Lemma \ref{lemmatildegestimates}, and the fact that $\tr{\tilde{\omega}}{\omega}$ and $\tr{\omega}{\tilde{\omega}}$ are uniformly equivalent.

Combining \eqref{x}, \eqref{eqnet2} and \eqref{dagger}, we have, at a point at which $Q$ achieves a maximum, for a uniform $C>0$,
\begin{align*}
\left( \ddt{} - \Delta \right) Q & \le  CA^2  + C e^{t/2} \tr{\omega}{\tilde{\omega}} + CAe^{t/2} - A e^{t/2} \tr{\omega}{\tilde{\omega}}
\end{align*}
where we are assuming, without loss of generality, that at this maximum point of $Q$ we have $\tr{\tilde{\omega}}{\omega} \ge 1$.  Choose a uniform $A$ large enough so that
$$A \ge C+1.$$
Then we obtain at the maximum of $Q$,
$$e^{t/2}\tr{\omega}{\tilde{\omega}}  \le CA^2 + CA e^{t/2},$$
which implies that $\tr{\omega}{\tilde{\omega}}$ and hence $\tr{\tilde{\omega}}{\omega}$ is uniformly bounded from above at the maximum of $Q$.  This establishes the estimate $\tr{\tilde{\omega}}{\omega} \le C$ and completes the proof of the theorem.
\end{proof}

\section{A bound on the Chern scalar curvature} \label{sectionscalar}

In this section, we establish the following estimate for the Chern scalar curvature:

\begin{theorem} \label{theoremscalar} There exists a uniform constant $C$ such that along the normalized Chern-Ricci flow \eqref{NCRF} we have
$$- C \le R \le Ce^{t/2},$$
for all $t\geq 0$.
\end{theorem}

First note that the lower bound for the Chern scalar curvature follows from the same argument as in the K\"ahler-Ricci flow (see for example  Theorem 2.2 in \cite{SW}). Indeed, from \eqref{NCRF}, we have
$$g^{k \ov{\ell}} \ddt{} g_{k\ov{\ell}} = - R - 2.$$
But $R = - g^{i\ov{j}} \partial_i \partial_{\ov{j}} \log \det g$ and hence
\begin{align*}
\frac{\de R}{\de t} ={} & - g^{i\ov{j}}  \partial_i \partial_{\ov{j}} \left( g^{k \ov{\ell}} \ddt{} g_{k\ov{\ell}} \right) - \left( \ddt{} g^{i\ov{j}} \right) \partial_i \partial_{\ov{j}} \log \det g \\
 = {} & \Delta R + |\textrm{Ric}|^2 + R \\
\ge {} & \Delta R + \frac{1}{2} R^2 + R
\end{align*}
and then the lower bound for $R$ follows.

We now establish the upper bound of the Chern scalar curvature.
Before we start the main argument, we need a few preliminary calculations.

\begin{lemma} \label{lemma2tr}  There exists a uniform constant $C>0$ such that for $t \ge T_I$, we have
\begin{equation}\label{nice}
\left( \ddt{} - \Delta \right) \tr{\ti{\omega}}{\omega}\leq -C^{-1}|\ti{\nabla}g|^2_g+Ce^{t/2}.
\end{equation}
and
\begin{equation}\label{nice2}
\left( \ddt{} - \Delta \right) \tr{\omega}{\omega_S}\leq |\ti{\nabla}g|^2_g-C^{-1}|\nabla \tr{\omega}{\omega_S}|^2_g+Ce^{t/2}.
\end{equation}
As a consequence, there are uniform positive constants $C_0, C_1$ such that for $t \ge T_I$,
\begin{equation}\label{nice3}
\left( \ddt{} - \Delta \right) (\tr{\omega}{\omega_S}+C_0\tr{\ti{\omega}}{\omega})\leq -|\ti{\nabla}g|^2_g-C_1^{-1}|\nabla \tr{\omega}{\omega_S}|^2_g+C_1e^{t/2}.
\end{equation}
\end{lemma}
\begin{proof}
For \eqref{nice}, we compute the evolution of $\tr{\ti{\omega}}{\omega}$.  As in Lemma \ref{lemmalogtr} we may modify  \cite[Proposition 3.1]{TW} to obtain, for $t \ge T_I$,
\[\begin{split}
\left( \ddt{} - \Delta \right) \tr{\ti{\omega}}{\omega}=&-g^{p\ov{j}}g^{i\ov{q}}\ti{g}^{k\ov{\ell}}\ti{\nabla}_k g_{i\ov{j}}\ti{\nabla}_{\ov{\ell}}g_{p\ov{q}}
- 2\mathrm{Re}\left(g^{i\ov{j}}\ti{g}^{k\ov{\ell}} \ti{T}^p_{ki}\ti{\nabla}_{\ov{\ell}}g_{p\ov{j}}\right) \\
&-g^{i\ov{j}}\ti{g}^{k\ov{\ell}}\ti{T}^p_{ik}\ov{\ti{T}^q_{j\ell}}g_{p\ov{q}}
+g^{i\ov{j}}\ti{g}^{k\ov{\ell}}(\ti{\nabla}_i\ov{\ti{T}^q_{j\ell}}-\ti{R}_{i\ov{\ell}p\ov{j}}\ti{g}^{p\ov{q}})g_{k\ov{q}}\\
&-g^{i\ov{j}}\ti{\nabla}_i   \ov{\ti{T}^\ell_{j\ell}}   -
 g^{i\ov{j}}\ti{g}^{k\ov{\ell}}\ti{g}_{p \ov{j}}\ti{\nabla}_{\ov{\ell}}  \ti{T}^p_{ik}
+g^{i\ov{j}}\ti{g}^{k\ov{\ell}} \ti{T}^p_{ik}\ov{\ti{T}^q_{j\ell}} \ti{g}_{p \ov{q}}
 -\tr{\ti{\omega}}{\omega}\\
 &-\ti{g}^{i\ov{\ell}}\ti{g}^{k\ov{j}}g_{i\ov{j}}((g_S)_{k\ov{\ell}}-\tilde{g}_{k\ov{\ell}}).
\end{split}\]
Using Theorem \ref{metricbdthm}, Lemma \ref{lemmatildegestimates}, and the Cauchy-Schwarz inequality in the second term, we obtain \eqref{nice}.

The inequality \eqref{nice2} is a parabolic Schwarz Lemma calculation for the map $\pi:M\to S$ \cite{ST, Ya}. Note that we already
know that $\tr{\omega}{\omega_S}\leq C$ since the metrics $\omega$ and $\tilde{\omega}$ are uniformly equivalent.

The computation for (\ref{nice2}) is similar to that of Song-Tian \cite{ST}, except that of course here we need to control the extra torsion terms.
Given any point $x\in M$ we choose local coordinates $\{z^i\}$ on $M$ centered at $x$ such that $g$ is the identity at $x$, and a coordinate $w$ on $S$ near $\pi(x)\in N$,
which we can assume is normal for the metric $g_S$.
In these coordinates we can represent
the map $\pi$ as a local holomorphic function $f$. We will use subscripts like $f_i, f_{ij},...$
to indicate covariant derivatives of $f$ with respect to $g$.
For example we have
$$f_i=\de_i f, \quad f_{ij}=\de_j f_i-\left(g^{k\ov{q}}\de_j g_{i\ov{q}}\right)
f_k, \quad f_{i\ov{j}}=f_{\ov{j}i}=0.$$
We will also write $g_S$ for the coefficient of the metric $g_S$ in the coordinate $w$, so $g_S(x)=1$ and the pullback of the
metric $g_S$ to $M$ is
given by $f_i \ov{f_j} g_S$. We  use the shorthand $h^{i\ov{j}}=g^{i\ov{\ell}}g^{k\ov{j}} f_k \ov{f_\ell} g_S,$ where
$h^{i\ov{j}}$ is semipositive definite and satisfies $|h|^2_g:= h^{i\ov{j}} h^{k\ov{\ell}} g_{i\ov{\ell}} g_{k\ov{j}} \leq C$.
Then we have (cf. \cite{To0})
\[\begin{split}
\Delta \tr{\omega}{\omega_S}=&g^{i\ov{j}}\de_i \de_{\ov{j}}\left(g^{k\ov{\ell}} f_k \ov{f_\ell} g_S\right)=g^{i\ov{j}}g^{k\ov{\ell}} f_{ki} \ov{f_{\ell j} } g_S+g^{i\ov{j}}h^{p\ov{q}} R_{i\ov{j}p\ov{q}}\\
&-g^{i\ov{j}} g^{k\ov{\ell}} f_k \ov{f_\ell} f_i \ov{f_j} R_S,
\end{split}\]
for $R_S$ the scalar curvature of $g_S$, and so
\begin{equation}\label{evtroos} \begin{split}
\left( \ddt{} - \Delta \right) \tr{\omega}{\omega_S} = & \tr{\omega}{\omega_S}-g^{i\ov{j}}g^{k\ov{\ell}} f_{ki} \ov{f_{\ell j} } g_S+g^{i\ov{j}}h^{p\ov{q}} (R_{p\ov{q}i\ov{j}}-R_{i\ov{j}p\ov{q}})\\
& +g^{i\ov{j}} g^{k\ov{\ell}} f_k \ov{f_\ell} f_i \ov{f_j} R_S.
\end{split}\end{equation}
The last term can be dropped since $R_S<0$.
Now at $x$ we have $\de_i (\tr{\omega}{\omega_S})=\sum_k f_{ki}\ov{f_k},$ and using the Cauchy-Schwarz inequality we have
\begin{equation}\label{cs0} \begin{split}
|\nabla \tr{\omega}{\omega_S}|^2_g = & \sum_{i,k,p} f_{ki}\ov{f_{pi}}f_{p}\ov{f_k}\\
\leq& \sum_{k,p}|f_k||f_p| \left( \sum_i |f_{ki}|^2\right)^{1/2} \left( \sum_j |f_{pj}|^2\right)^{1/2}\\
= &\left(\sum_k |f_k| \left(\sum_i |f_{ki}|^2 \right)^{1/2}\right)^2\leq \left(\sum_\ell |f_\ell|^2\right)
\left(\sum_{i, k}|f_{ki}|^2\right)\\
=  & (\tr{\omega}{\omega_S})g^{i\ov{j}}g^{k\ov{\ell}} f_{ki} \ov{f_{\ell j} } g_S\leq Cg^{i\ov{j}}g^{k\ov{\ell}} f_{ki} \ov{f_{\ell j} } g_S.
\end{split}\end{equation}
Next we claim that given any constant $C_0$ we can find a constant $C$ such that
\begin{equation}\label{sette0}
|g^{i\ov{j}}h^{p\ov{q}}(R_{p\ov{q}i\ov{j}}-R_{i\ov{j}p\ov{q}})|\leq Ce^{t/2}+\frac{1}{2C_0}|\ti{\nabla} g|^2_g.
\end{equation}
In fact, we only need the case $C_0=1$ here, but the general case will be useful later.
To prove this claim, we first calculate (see also \cite[(2.6)]{ShW}),
\begin{equation}\label{switch0}
\begin{split}
R_{i\ov{j}p\ov{q}}&=-g_{r\ov{q}} \de_{\ov{j}} \Gamma^r_{ip}=
-g_{r\ov{q}} \de_{\ov{j}} \Gamma^r_{pi}+g_{r\ov{q}} \de_{\ov{j}} T^r_{pi}=
R_{p\ov{j}i\ov{q}}+g_{r\ov{q}} \de_{\ov{j}} T^r_{pi}\\
&=\ov{R_{j\ov{p}q\ov{i}}}+g_{r\ov{q}} \de_{\ov{j}} T^r_{pi}
=\ov{R_{q\ov{p}j\ov{i}}}+g_{i\ov{s}} \de_{p} \ov{T^s_{qj}}+g_{r\ov{q}} \de_{\ov{j}} T^r_{pi}\\
&=R_{p\ov{q}i\ov{j}}+g_{i\ov{s}} \de_{p} \ov{T^s_{qj}}+g_{r\ov{q}} \de_{\ov{j}} T^r_{pi},
\end{split}\end{equation}
and therefore
\begin{equation}\label{tre0}
g^{i\ov{j}}h^{p\ov{q}}(R_{p\ov{q}i\ov{j}}-R_{i\ov{j}p\ov{q}})=-h^{p\ov{q}}\de_{p} \ov{T^j_{qj}} -g^{i\ov{j}}h^{p\ov{q}}g_{r\ov{q}} \de_{\ov{j}} T^r_{pi}.
\end{equation}
Recall that $T_{ij\ov{\ell}}=\ti{T}_{ij\ov{\ell}}.$ Differentiating this gives
\begin{equation}\label{quattro0}
\de_p \ov{T^j_{qj}}=\ti{g}_{r\ov{s}}g^{r\ov{j}} \de_p \ov{\ti{T}^s_{qj}}-\ti{g}_{r\ov{s}} g^{r\ov{u}}g^{t\ov{j}}\ov{\ti{T}^s_{qj}}
\ti{\nabla}_p g_{t\ov{u}}=g^{r\ov{j}} \ti{\nabla}_p \ov{\ti{T}_{qj\ov{r}}}- g^{r\ov{u}}g^{t\ov{j}}\ov{\ti{T}_{qj\ov{r}}}
\ti{\nabla}_p g_{t\ov{u}},
\end{equation}
\begin{equation}\label{cinque0}
\de_{\ov{j}}T^r_{pi}=\ti{g}_{s\ov{b}}g^{r\ov{b}}\de_{\ov{j}}\ti{T}^s_{pi}-\ti{g}_{s\ov{b}}g^{r\ov{u}}g^{t\ov{b}} \ti{T}^s_{pi} \ti{\nabla}_{\ov{j}}g_{t\ov{u}}
=g^{r\ov{b}}\ti{\nabla}_{\ov{j}}\ti{T}_{pi\ov{b}}-g^{r\ov{u}}g^{t\ov{b}}\ti{T}_{pi\ov{b}}\ti{\nabla}_{\ov{j}}g_{t\ov{u}}.
\end{equation}
Putting together \eqref{tre0}, \eqref{quattro0}, \eqref{cinque0} we get
\begin{equation}\label{sei0}
\begin{split}
g^{i\ov{j}}h^{p\ov{q}}(R_{p\ov{q}i\ov{j}}-R_{i\ov{j}p\ov{q}})&=-h^{p\ov{q}}g^{r\ov{j}} \ti{\nabla}_p \ov{\ti{T}_{qj\ov{r}}} +h^{p\ov{q}} g^{r\ov{u}}g^{t\ov{j}}\ov{\ti{T}_{qj\ov{r}}}
\ti{\nabla}_p g_{t\ov{u}}\\
&-h^{p\ov{q}}g^{i\ov{j}}\ti{\nabla}_{\ov{j}}\ti{T}_{pi\ov{q}}+h^{p\ov{q}}g^{i\ov{j}}g^{t\ov{r}}\ti{T}_{pi\ov{r}}\ti{\nabla}_{\ov{j}}g_{t\ov{q}}.
\end{split}\end{equation}
Using again that the metrics $g$ and $\ti{g}$ are equivalent and $|h|_g \le C$, we can then bound this by
$$|g^{i\ov{j}}h^{p\ov{q}}(R_{p\ov{q}i\ov{j}}-R_{i\ov{j}p\ov{q}})|\leq C|\db \ti{T}|_{\ti{g}}+C|\ti{T}|_{\ti{g}}|\ti{\nabla} g|_g.$$
But from  Lemma \ref{lemmatildegestimates} we have that $|\ti{T}|_{\ti{g}}\leq C$ and $|\db \ti{T}|_{\ti{g}}\leq Ce^{t/2}$, and so we have \eqref{sette0} as required.

Then \eqref{nice2} follows from \eqref{evtroos}, \eqref{cs0} and \eqref{sette0}.   \eqref{nice3} follows immediately from \eqref{nice} and \eqref{nice2}.
\end{proof}

We can now start the main argument for the proof of Theorem \ref{theoremscalar}.  Note that since many of our inequalities require $\tilde{\omega}$ to be a metric, we will often assume (without saying it explicitly) that $t \ge T_I$, which is not a problem since $R$ is bounded on $[0,T_I]$.
 As in \cite{ST}, we consider the quantity
 $u=\vp+\dot{\vp}=\log\frac{e^{t}\omega^2}{\Omega}.$ We know that $|u|\leq C$, and we have that
$-\Delta u =R+\tr{\omega}{\omega_S}\geq R$, so our goal is to get an upper bound for $-\Delta u$. First compute from (\ref{eveqndotphi}),
\begin{equation} \label{evolveu0}
\left( \ddt{} - \Delta \right)u=\tr{\omega}{\omega_S}-1,
\end{equation}
and
$$\left( \ddt{} - \Delta \right)\Delta u=R^{i\ov{j}}u_{i\ov{j}}+\Delta u+\Delta\tr{\omega}{\omega_S}.$$
But on the other hand $R_{i\ov{j}}=-u_{i\ov{j}}-(g_S)_{i\ov{j}},$ and so
\[\begin{split}
\left( \ddt{} - \Delta \right)\Delta u&=-|\nabla \ov{\nabla} u|^2_g-\langle g_S, \nabla \ov{\nabla} u\rangle_g+\Delta u+\Delta\tr{\omega}{\omega_S}\\
&\geq -\frac{3}{2}|\nabla \ov{\nabla} u|^2_g+\Delta u+\Delta\tr{\omega}{\omega_S}-C,
\end{split}\]
using that $|g_S|_g\leq C$.
From (\ref{nice2}) we have
\begin{align*}
-\Delta\tr{\omega}{\omega_S}  \le {} &  Ce^{t/2}+|\ti{\nabla} g|^2_g-C^{-1}|\nabla \tr{\omega}{\omega_S}|^2_g-h^{i\ov{j}} (R_{i\ov{j}}+ g_{i\ov{j}})\\
\le {} &Ce^{t/2}+|\ti{\nabla} g|^2_g-C^{-1}|\nabla \tr{\omega}{\omega_S}|^2_g+h^{i\ov{j}}(u_{i\ov{j}}+(g_S)_{i\ov{j}})\\
\leq {} &Ce^{t/2}+|\ti{\nabla} g|^2_g-C^{-1}|\nabla \tr{\omega}{\omega_S}|^2_g+\frac{1}{2}|\nabla \ov{\nabla} u|^2_g,
\end{align*}
where we have used the fact that $\ddt{} \tr{\omega}{\omega_S} = -h^{i\ov{j}} \ddt{} g_{i\ov{j}}$, with $h^{i\ov{j}}$ as in the proof of Lemma \ref{lemma2tr}.
Therefore
\begin{equation}\label{otto}
\begin{split}
\left( \ddt{} - \Delta \right)(-\Delta u)&\leq 2|\nabla \ov{\nabla} u|^2_g-\Delta u+Ce^{t/2}+|\ti{\nabla} g|^2_g-C^{-1}|\nabla \tr{\omega}{\omega_S}|^2_g.
\end{split}
\end{equation}
We need a quantity whose evolution can kill the bad term $2|\nabla \ov{\nabla} u|^2_g$, and this quantity is $|\nabla u|^2_g$.
Before we compute its evolution, we need formulae to commute two covariant derivatives of the same type.
For any function $\psi$ and $(0,1)$ form $a=a_{\ov{k}}d\ov{z}^k$, a short calculation gives
$$[\nabla_{\ov{j}},\nabla_{\ov{\ell}}]\psi=-\ov{T_{j\ell}^k}\nabla_{\ov{k}}\psi, \quad [\nabla_i,\nabla_j]a_{\ov{k}}=-T^\ell_{ij}\nabla_\ell a_{\ov{k}}.$$
We will also use the familiar formulae
$$[\nabla_i,\nabla_{\ov{j}}]a_{\ov{k}}=g^{p\ov{\ell}}R_{i\ov{j}p\ov{k}}a_{\ov{\ell}}, \quad R_{i\ov{\ell}p\ov{j}}=R_{p\ov{\ell}i\ov{j}}+g_{r\ov{j}}\de_{\ov{\ell}}T^r_{pi},$$
where the second equation is contained in \eqref{switch0}.
We then compute:
\[\begin{split}
\Delta &|\nabla u|^2_g  \\
=&g^{i\ov{j}} g^{k\ov{\ell}}\bigg(\nabla_i \nabla_{\ov{j}} \nabla_k u \nabla_{\ov{\ell}} u+\nabla_k u \nabla_i \nabla_{\ov{j}}\nabla_{\ov{\ell}} u
+\nabla_i\nabla_k u \nabla_{\ov{j}}\nabla_{\ov{\ell}} u+ \nabla_i \nabla_{\ov{\ell}} u \nabla_{\ov{j}}\nabla_k u \bigg)\\
=&|\nabla \ov{\nabla} u|^2_g+|\nabla \nabla u|^2_g+g^{i\ov{j}} g^{k\ov{\ell}}\bigg(\nabla_i \nabla_k \nabla_{\ov{j}}  u \nabla_{\ov{\ell}} u+\nabla_k u \nabla_i \nabla_{\ov{\ell}}\nabla_{\ov{j}} u\\
&-\nabla_k u \nabla_i(\ov{T^p_{j\ell}}\nabla_{\ov{p}}u)\bigg)\\
=&|\nabla \ov{\nabla} u|^2_g+|\nabla \nabla u|^2_g+2\mathrm{Re}\langle \nabla \Delta u,\nabla u\rangle_g
-g^{i\ov{j}} g^{k\ov{\ell}}T_{ik}^p \nabla_p \nabla_{\ov{j}}u \nabla_{\ov{\ell}} u\\
&+g^{i\ov{j}} g^{k\ov{\ell}} g^{p\ov{q}}R_{i\ov{\ell}p\ov{j}}\nabla_{\ov{q}}u \nabla_k u-g^{i\ov{j}} g^{k\ov{\ell}}\ov{T^p_{j\ell}}\nabla_k u \nabla_i\nabla_{\ov{p}}u-g^{i\ov{j}} g^{k\ov{\ell}}\de_i\ov{T^p_{j\ell}}\nabla_k u\nabla_{\ov{p}}u  \\
=&|\nabla \ov{\nabla} u|^2_g+|\nabla \nabla u|^2_g+2\mathrm{Re}\langle \nabla \Delta u,\nabla u\rangle_g+g^{k\ov{\ell}} g^{p\ov{q}}R_{p\ov{\ell}}\nabla_{\ov{q}}u \nabla_k u\\
&-2\mathrm{Re}\left(g^{i\ov{j}} g^{k\ov{\ell}}T_{ik}^p \nabla_p \nabla_{\ov{j}}u \nabla_{\ov{\ell}} u\right)+ g^{k\ov{\ell}} g^{p\ov{q}}\de_{\ov{\ell}}T^i_{pi}\nabla_{\ov{q}}u \nabla_k u
-g^{i\ov{j}} g^{k\ov{\ell}}\de_i\ov{T^p_{j\ell}}\nabla_k u\nabla_{\ov{p}}u .
\end{split}\]
We use again \eqref{quattro0} and \eqref{cinque0}
\[
\de_{i}\ov{T^p_{j\ell}}=g^{q\ov{p}}\ti{\nabla}_{i}\ov{\ti{T}_{j\ell\ov{q}}}-g^{r\ov{p}}g^{q\ov{s}}\ov{\ti{T}_{j\ell\ov{q}}}\ti{\nabla}_{i}g_{r\ov{s}}, \quad
\de_{\ov{\ell}}T^i_{pi}=g^{i\ov{j}}\ti{\nabla}_{\ov{\ell}}\ti{T}_{pi\ov{j}}-g^{i\ov{s}}g^{r\ov{j}}\ti{T}_{pi\ov{j}}\ti{\nabla}_{\ov{\ell}}g_{r\ov{s}},\]
and Lemma \ref{lemmatildegestimates} to conclude that
\[\begin{split}
\Delta |\nabla u|^2_g \geq  & |\nabla \ov{\nabla} u|^2_g+|\nabla \nabla u|^2_g+2\mathrm{Re}\langle \nabla \Delta u,\nabla u\rangle_g+g^{k\ov{\ell}} g^{p\ov{q}}R_{p\ov{\ell}}\nabla_{\ov{q}}u \nabla_k u\\
&-C|\nabla\ov{\nabla}u|_g |\nabla u|_g-Ce^{t/2}|\nabla u|^2_g-C|\ti{\nabla}g|_g |\nabla u|^2_g\\
\geq  & \frac{1}{2}|\nabla \ov{\nabla} u|^2_g+|\nabla \nabla u|^2_g+2\mathrm{Re}\langle \nabla \Delta u,\nabla u\rangle_g+g^{k\ov{\ell}} g^{p\ov{q}}R_{p\ov{\ell}}\nabla_{\ov{q}}u \nabla_k u\\
&-Ce^{t/2}|\nabla u|^2_g-C|\ti{\nabla}g|_g |\nabla u|^2_g.
\end{split}\]
Next, using (\ref{evolveu0}), we have
$$\ddt{} | \nabla u|^2_g = g^{k\ov{\ell}} g^{p\ov{q}} R_{p\ov{\ell}} \nabla_k u \nabla_{\ov{q}}u + | \nabla u|^2_g + 2 \textrm{Re} \langle \nabla \Delta u, \nabla u \rangle_g + 2 \textrm{Re} \langle \nabla \tr{\omega}{\omega_S}, \nabla u \rangle_g$$
and hence
\begin{equation}\label{nove}
\begin{split}
\left( \ddt{} - \Delta \right)|\nabla u|^2_g \leq  & -\frac{1}{2}|\nabla \ov{\nabla} u|^2_g-|\nabla \nabla u|^2_g+2\mathrm{Re}\langle \nabla \tr{\omega}{\omega_S},\nabla u\rangle_g\\
&+Ce^{t/2}|\nabla u|^2_g+C|\ti{\nabla}g|_g |\nabla u|^2_g.
\end{split}
\end{equation}
We will use this evolution inequality to bound $|\nabla u|^2_g$.

\begin{proposition}
There is a constant $C$ such that
$$|\nabla u|^2_g \leq Ce^{t/2}.$$
\end{proposition}
\begin{proof}
We use the method of Cheng-Yau \cite{CY}, see also \cite{SeT, ST}. We fix a constant $A$ such that $|u|\leq A-1$ and use \eqref{nove} to compute
\begin{equation}\label{comp1}
\begin{split}
\left( \ddt{} - \Delta \right)&\left(\frac{|\nabla u|^2_g}{A-u}\right)\\
= &\frac{1}{A-u}\left( \ddt{} - \Delta \right)|\nabla u|^2_g
+\frac{|\nabla u|^2_g}{(A-u)^2}\left( \ddt{} - \Delta \right)u\\
&-\frac{2\mathrm{Re}\langle \nabla|\nabla u|^2_g,\nabla u\rangle_g}{(A-u)^2}
-2\frac{|\nabla u|^4_g}{(A-u)^3}\\
\leq  & \frac{1}{A-u}\bigg( -\frac{1}{2}|\nabla \ov{\nabla} u|^2_g-|\nabla \nabla u|^2_g
+2\mathrm{Re}\langle \nabla \tr{\omega}{\omega_S},\nabla u\rangle_g\\
&+Ce^{t/2}|\nabla u|^2_g+C|\ti{\nabla}g|_g |\nabla u|^2_g\bigg)
+(\tr{\omega}{\omega_S}-1)\frac{|\nabla u|^2_g}{(A-u)^2}\\
&-\frac{2\mathrm{Re}\langle \nabla|\nabla u|^2_g,\nabla u\rangle_g}{(A-u)^2}-2\frac{|\nabla u|^4_g}{(A-u)^3}.
\end{split}
\end{equation}
Note that the term $(\tr{\omega}{\omega_S}-1)\frac{|\nabla u|^2_g}{(A-u)^2}$ can be absorbed in the term
$\frac{Ce^{t/2}|\nabla u|^2_g}{A-u}$.
For $\ve>0$ small (to be fixed soon), write
\begin{equation}\label{comp2}
\begin{split}
\frac{2\mathrm{Re}\langle \nabla|\nabla u|^2_g,\nabla u\rangle_g}{(A-u)^2}=&\ve\frac{2\mathrm{Re}\langle \nabla|\nabla u|^2_g,\nabla u\rangle_g}{(A-u)^2}\\
&+\frac{2(1-\ve)}{A-u}\mathrm{Re}\left\langle \nabla\left(\frac{|\nabla u|^2_g}{A-u}\right),\nabla u\right\rangle_g\\
&-\frac{2(1-\ve)|\nabla u|^4_g}{(A-u)^3},
\end{split}
\end{equation}
and use the Cauchy-Schwarz inequality to bound
\begin{equation}\label{comp3}
\begin{split}
-\ve\frac{2\mathrm{Re}\langle \nabla|\nabla u|^2_g,\nabla u\rangle_g}{(A-u)^2}&\leq \ve 2\sqrt{2} \frac{|\nabla u|^2_g (|\nabla \ov{\nabla}u|^2_g+|\nabla\nabla u|^2_g)^{1/2}}{(A-u)^2}\\
&\leq \frac{\ve}{2} \frac{|\nabla u|^4_g}{(A-u)^2} + 4\ve  \frac{|\nabla \ov{\nabla}u|^2_g+|\nabla\nabla u|^2_g}{(A-u)^2} \\
&\leq \frac{\ve}{2} \frac{|\nabla u|^4_g}{(A-u)^2} + \frac{1}{2}  \frac{|\nabla \ov{\nabla}u|^2_g+|\nabla\nabla u|^2_g}{A-u},
\end{split}
\end{equation}
provided $\ve\leq 1/8$ (here we fix $\ve$).
We also bound
\begin{equation}\label{comp5}
\frac{C|\ti{\nabla}g|_g |\nabla u|^2_g}{A-u}\leq C|\ti{\nabla}g|^2_g+\frac{\ve}{2}\frac{|\nabla u|^4_g}{(A-u)^2}.
\end{equation}
Putting \eqref{comp1}, \eqref{comp2}, \eqref{comp3}, \eqref{comp5} together, we get
\[\begin{split}
\left( \ddt{} - \Delta \right)\left(\frac{|\nabla u|^2_g}{A-u}\right)
\leq &\frac{1}{A-u}\bigg(2\mathrm{Re}\langle \nabla \tr{\omega}{\omega_S},\nabla u\rangle_g+Ce^{t/2}|\nabla u|^2_g\bigg)+C|\ti{\nabla}g|^2_g\\
&-\ve\frac{|\nabla u|^4_g}{(A-u)^3}-\frac{2(1-\ve)}{A-u}\mathrm{Re}\left\langle \nabla\left(\frac{|\nabla u|^2_g}{A-u}\right),\nabla u\right\rangle_g.
\end{split}\]
Call now
$$Q_1=\frac{|\nabla u|^2_g}{A-u}+C_2(\tr{\omega}{\omega_S}+C_0\tr{\ti{\omega}}{\omega}),$$
where $C_0$ is as in \eqref{nice3}, and $C_2$ is a large uniform constant to be fixed soon.
We can use \eqref{nice3} to get
\begin{equation}\label{comp6}
\begin{split}
\left( \ddt{} - \Delta \right)Q_1
\leq& \frac{1}{A-u}\bigg(2\mathrm{Re}\langle \nabla \tr{\omega}{\omega_S},\nabla u\rangle_g+Ce^{t/2}|\nabla u|^2_g\bigg)\\
&-\ve\frac{|\nabla u|^4_g}{(A-u)^3}-\frac{2(1-\ve)}{A-u}\mathrm{Re}\left\langle \nabla\left(\frac{|\nabla u|^2_g}{A-u}\right),\nabla u\right\rangle_g\\
&-\frac{C_2}{2}|\ti{\nabla}g|^2_g-2|\nabla \tr{\omega}{\omega_S}|^2_g+Ce^{t/2},
\end{split}
\end{equation}
provided $C_2$ is large enough.
We now observe that $|\nabla \tr{\ti{\omega}}{\omega}|^2_{\ti{g}}\leq 2|\ti{\nabla} g|^2_{\ti{g}}.$
Indeed, if we choose local coordinates such that $\ti{g}$ is the identity at a point, then at that point we have
\begin{equation}\label{cs}
|\nabla \tr{\ti{\omega}}{\omega}|^2_{\ti{g}}=\sum_k |\sum_i \ti{\nabla}_k g_{i\ov{i}}|^2\leq
2\sum_{i,k}|\ti{\nabla}_k g_{i\ov{i}}|^2\leq 2\sum_{i,j,k}|\ti{\nabla}_k g_{i\ov{j}}|^2=2|\ti{\nabla} g|^2_{\ti{g}}.
\end{equation}
Since $\ti{g}$ and $g$ are uniformly equivalent, we conclude that
$|\nabla \tr{\ti{\omega}}{\omega}|^2_{g}\leq C|\ti{\nabla} g|^2_{g}.$
We can then assume that $C_2$ was large enough (this fixes $C_2$) so that
\begin{equation}\label{comp7}
-\frac{C_2}{2}|\ti{\nabla} g|^2_{g}\leq - |\nabla \tr{\ti{\omega}}{\omega}|^2_{g}.
\end{equation}
Furthermore, we can bound
\begin{equation}\label{comp8}
C\frac{e^{t/2}|\nabla u|^2_g}{A-u}\leq \frac{\ve}{4}\frac{|\nabla u|^4_g}{(A-u)^3}+Ce^t,
\end{equation}
\begin{equation}\label{comp9}
2\frac{\mathrm{Re}\langle \nabla \tr{\omega}{\omega_S},\nabla u\rangle_g}{A-u}\leq |\nabla \tr{\omega}{\omega_S}|^2_g+\frac{\ve}{4}\frac{|\nabla u|^4_g}{(A-u)^3}+C,
\end{equation}
so combining \eqref{comp6}, \eqref{comp7}, \eqref{comp8}, \eqref{comp9}  we get
\[\begin{split}
\left( \ddt{} - \Delta \right)Q_1
\leq& -\frac{\ve}{2}\frac{|\nabla u|^4_g}{(A-u)^3}-\frac{2(1-\ve)}{A-u}\mathrm{Re}\left\langle \nabla\left(\frac{|\nabla u|^2_g}{A-u}\right),\nabla u\right\rangle_g\\
& - |\nabla \tr{\ti{\omega}}{\omega}|^2_{g}-|\nabla \tr{\omega}{\omega_S}|^2_g+Ce^{t}.
\end{split}\]
We can write this as
\[\begin{split}
\left( \ddt{} - \Delta \right)Q_1
\leq&-\frac{\ve}{2}\frac{|\nabla u|^4_g}{(A-u)^3}-\frac{2(1-\ve)}{A-u}\mathrm{Re}\left\langle \nabla Q_1,\nabla u\right\rangle_g\\
& +\frac{2(1-\ve)C_2}{A-u}\mathrm{Re}\left\langle \nabla \tr{\omega}{\omega_S},\nabla u\right\rangle_g\\
&+\frac{2(1-\ve)C_0C_2}{A-u}\mathrm{Re}\left\langle \nabla \tr{\ti{\omega}}{\omega},\nabla u\right\rangle_g\\
& - |\nabla \tr{\ti{\omega}}{\omega}|^2_{g}-|\nabla \tr{\omega}{\omega_S}|^2_g+Ce^{t},
\end{split}\]
which together with the bounds
$$\frac{2(1-\ve)C_2}{A-u}\mathrm{Re}\left\langle \nabla \tr{\omega}{\omega_S},\nabla u\right\rangle_g\leq |\nabla \tr{\omega}{\omega_S}|^2_g+\frac{\ve}{8}\frac{|\nabla u|^4_g}{(A-u)^3}+C,$$
$$\frac{2(1-\ve)C_0C_2}{A-u}\mathrm{Re}\left\langle \nabla \tr{\ti{\omega}}{\omega},\nabla u\right\rangle_g\leq  |\nabla \tr{\ti{\omega}}{\omega}|^2_g+\frac{\ve}{8}\frac{|\nabla u|^4_g}{(A-u)^3}+C,$$
gives us
\[\begin{split}
\left( \ddt{} - \Delta \right)Q_1&
\leq-\frac{\ve}{4}\frac{|\nabla u|^4_g}{(A-u)^3}-\frac{2(1-\ve)}{A-u}\mathrm{Re}\left\langle \nabla Q_1,\nabla u\right\rangle_g+Ce^{t}.
\end{split}\]
Now define $Q_2=e^{-t/2}Q_1$, which satisfies
\[\begin{split}
\left( \ddt{} - \Delta \right)Q_2&
\leq-C^{-1}e^{-t/2}|\nabla u|^4_g-\frac{2(1-\ve)}{A-u}\mathrm{Re}\left\langle \nabla Q_2,\nabla u\right\rangle_g+Ce^{t/2}.
\end{split}\]
At a maximum of $Q_2$ (occurring at a time $t_0 >  T_I$) we conclude that $|\nabla u|^4_g\leq Ce^t$, which implies that $Q_2\leq C$ everywhere. This proves the proposition.
\end{proof}
Now that we know that $|\nabla u|^2_g \leq Ce^{t/2}$, we can go back to \eqref{nove} and get
\begin{equation}\label{dieci}
\begin{split}
\left( \ddt{} - \Delta \right)|\nabla u|^2_g&\leq -\frac{1}{2}|\nabla \ov{\nabla} u|^2_g-|\nabla \nabla u|^2_g+
|\nabla \tr{\omega}{\omega_S}|^2_g+|\ti{\nabla}g|^2_g+Ce^t.
\end{split}
\end{equation}

Finally we can put everything together to prove Theorem \ref{theoremscalar}.

\begin{proof}[Proof of Theorem \ref{theoremscalar}]
We will show that
there is a constant $C$ such that
$$-\Delta u\leq Ce^{t/2}.$$
and this will give $R\leq Ce^{t/2}$.

From \eqref{otto} and \eqref{dieci}, we see that
\[\begin{split}
\left( \ddt{} - \Delta \right)(-\Delta u+6|\nabla u|^2_g)\leq& -|\nabla \ov{\nabla} u|^2_g-\Delta u+
C|\nabla \tr{\omega}{\omega_S}|^2_g\\
&+C|\ti{\nabla}g|^2_g+Ce^t.
\end{split}\]
Using \eqref{nice3} we get
\[\begin{split}
\left( \ddt{} - \Delta \right)\left(-\Delta u+6|\nabla u|^2_g+C_2(\tr{\omega}{\omega_S}+C_0\tr{\ti{\omega}}{\omega})\right)&\leq -|\nabla \ov{\nabla} u|^2_g-\Delta u+Ce^t,
\end{split}\]
provided $C_2$ is large enough. Define now
$$Q_3=e^{-t/2}\left(-\Delta u+6|\nabla u|^2_g+C_2(\tr{\omega}{\omega_S}+C_0\tr{\ti{\omega}}{\omega})\right).$$
Note that $Q_3\geq -Ce^{-t/2}$, because $-\Delta u\geq R\geq -C$.
Then,
$$\left( \ddt{} - \Delta \right)Q_3\leq -e^{-t/2}|\nabla \ov{\nabla} u|^2_g-e^{-t/2}\Delta u+Ce^{t/2},$$
where we absorbed a term like $e^{-t/2}$ into $e^{t/2}$. From the  Cauchy-Schwarz inequality we also have that
$$(-\Delta u)^2\leq 2|\nabla\ov{\nabla}u|^2_g.$$
It follows that at a maximum of $Q_3$ (occurring at $t_0 > T_I$) we have that
$(-\Delta u)^2\leq Ce^t$, which implies that $Q_3\leq C$ everywhere. This completes the proof of the theorem.
\end{proof}

We end this section by applying our Chern scalar curvature bound to obtain an exponential decay estimate for $\dot{\varphi}$.

\begin{lemma} \label{expdotphi}  For any $\eta$ with $0< \eta <1/2$ and any $\sigma$ with $0<\sigma <1/4$, there exists a constant $C$ such that
$$-C e^{-\eta t} \le \dot{\varphi} \le C e^{-\sigma t}.$$
\end{lemma}
\begin{proof}
We first prove the lower bound by refining an argument  in \cite{SW}.  We have
\begin{equation} \label{evolvedp0}
\ddt{} \dot{\varphi} = -R - 1 -\dot{\varphi}
\end{equation}
and hence, by Theorem \ref{theoremscalar} and the fact that $|\dot{\varphi}|$ is bounded,
\begin{equation} \label{dpC00}
\ddt{} \dot{\varphi}(t) \le  C_0,
\end{equation}
for a uniform $C_0$.
Suppose for a contradiction that we do not have the bound $\dot{\varphi} \ge -Ce^{-\eta t}$ for any $C$.  Then there exists a sequence $(x_k, t_k) \in M \times [0,\infty)$ with $t_k \rightarrow \infty$ as $k \rightarrow \infty$ such that
$$\dot{\varphi}(x_k, t_k) \le -k e^{-\eta t_k}.$$
Put $\gamma_k = \frac{k}{2C_0} e^{-\eta t_k}$.  From now on we work at the point $x_k$.
Then by \eqref{dpC00}, we have that
$$\dot{\varphi} \le -\frac{k}{2} e^{-\eta t_k} \quad \textrm{on } [t_k, t_k + \gamma_k ].$$
Indeed,
$$\dot{\varphi}(t_k+ a) - \dot{\varphi}(t_k) = \int_{t_k}^{t_k +a} \ddt{} \dot{\varphi} \, dt \le   C_0  \gamma_k , \quad \textrm{for } a \in [0,\gamma_k],$$
and hence $\dot{\varphi}(t_k+a) \le \dot{\varphi}(t_k) + C_0 \gamma_k   \le - ke^{-\eta t_k} + \frac{k}{2} e^{-\eta t_k}$.

Then, using Lemma \ref{key},
$$- C (1+t_k) e^{-t_k} \le \varphi(t_k+\gamma_k) - \varphi(t_k) = \int_{t_k}^{t_k+\gamma_k} \dot{\varphi} dt \le - \gamma_k \frac{k}{2} e^{-\eta t_k} = - \frac{k^2}{4 C_0} e^{-2\eta t_k}.$$
But if $2\eta <1$ then we get a contradiction when $k \rightarrow \infty$ and we are done.

For the upper bound of $\varphi$ we use the upper bound $R \le Ce^{t/2}$ of Theorem \ref{theoremscalar}.  From \eqref{evolvedp0}, we have
\begin{equation}
\ddt{} \dot{\varphi}(t) \ge - C_0 e^{t/2},
\end{equation}
for a uniform $C_0$.  Note that we may assume, by increasing $C_0$, that we have
\begin{equation} \label{dotphi10}
\frac{\partial^2}{\partial t^2} \varphi (t) \ge - C_0 e^{t'/2} \quad \textrm{for } t \in [t', t'+1],
\end{equation}
for any time $t'$.
Suppose for a contradiction that we do not have the bound $\dot{\varphi} \le Ce^{-\sigma t}$ for any $C$.  Then there exists a sequence $(x_k, t_k) \in M \times [0,\infty)$ with $t_k \rightarrow \infty$ as $k \rightarrow \infty$ such that
$$\dot{\varphi}(x_k, t_k) \ge k e^{-\sigma t_k}.$$
Put $\gamma_k = \frac{k}{2C_0} e^{-(\sigma + 1/2) t_k}$ which we assume for the moment satisfies $\gamma_k \le 1$.  From now on we work at the point $x_k$.
Then by \eqref{dotphi10}, we have that
$$\dot{\varphi} \ge \frac{k}{2} e^{-\sigma t_k} \quad \textrm{on } [t_k, t_k + \gamma_k ].$$
Indeed, this follows from (\ref{dotphi10}) since
$$\dot{\varphi}(t_k+ a) - \dot{\varphi}(t_k) = \int_{t_k}^{t_k+a} \ddt{} \dot{\varphi} \, dt \ge - \gamma_k C_0 e^{t_k/2}, \quad \textrm{for } a \in [0,\gamma_k],$$
and hence $\dot{\varphi}(t_k+a) \ge \dot{\varphi}(t_k) - \gamma_k C_0 e^{t_k/2} \ge ke^{-\sigma t_k} - \frac{k}{2} e^{-\sigma t_k}$.

Then from Lemma \ref{key},
$$C (1+t_k) e^{-t_k} \ge \varphi(t_k+\gamma_k) - \varphi(t_k) = \int_{t_k}^{t_k+\gamma_k} \dot{\varphi} dt \ge \gamma_k \frac{k}{2} e^{-\sigma t_k} = \frac{k^2}{4 C_0} e^{-(2\sigma+1/2)t_k}.$$
But since $2\sigma + 1/2 <1$ we get a contradiction when $k \rightarrow \infty$ and we are done.

It remains to check the case when $\gamma_k >1$.  But then we have $\dot{\varphi} \ge \frac{k}{2} e^{-\sigma t_k}$ on $[t_k, t_k+1]$ since $\dot{\varphi}(t_k +a) - \dot{\varphi}(t_k)  \ge - \gamma_k C_0 e^{t_k/2}$ for $a\in [0,1]$ and so we  get $\dot{\varphi}(t_k+a) \ge \frac{k}{2} e^{-\sigma t_k}$ for $a \in [0, 1]$.  Then
$$C (1+t_k) e^{-t_k} \ge \varphi(t_k+1) - \varphi(t_k) = \int_{t_k}^{t_k+1} \dot{\varphi} dt \ge  \frac{k}{2} e^{-\sigma t_k}$$
and we get a contradiction since $\sigma<1$.
\end{proof}

\section{Exponential decay estimates for the  metric} \label{sectionexp}

In this section we establish the key estimates which show that $\omega(t)$ and $\tilde{\omega}$ approach each other exponentially fast as $t \rightarrow \infty$.  More precisely we prove:

\begin{theorem} \label{theorempinch}  For any $\ve$ with $0<\ve < 1/8$, there exists $C$ such that for $t\geq T_I$
$$(1- Ce^{-\ve t}) \tilde{\omega} \le \omega(t) \le (1+Ce^{-\ve t}) \tilde{\omega}.$$
\end{theorem}

In this section we will always assume $t\geq T_I$, without necessarily mentioning it explicitly.
First,  we have the following  evolution inequality for $\tr{\omega}{\tilde{\omega}}$ which has the same form as the inequality for $\tr{\tilde{\omega}}{\omega}$ given by (\ref{nice}).  We will make use of both of these inequalities to prove Theorem \ref{theorempinch}.

\begin{lemma}  \label{schwarztype} We have, for a uniform $C>0$,
$$\left( \ddt{} - \Delta \right) \tr{\omega}{\tilde{\omega}} \le Ce^{t/2}-C^{-1}|\ti{\nabla} g|^2_g.$$
\end{lemma}
\begin{proof}
We will use the shorthand $h^{i\ov{j}}=g^{i\ov{\ell}}g^{k\ov{j}} \ti{g}_{k\ov{\ell}}$.
 Note that we know already that $h, g$ and $\ti{g}$ are
all uniformly equivalent to each other.

We start by computing a formula for the evolution of $\tr{\omega}{\ti{\omega}}$. First of all, we have

\[\begin{split}
\Delta \tr{\omega}{\ti{\omega}}=&g^{i\ov{j}}\ti{\nabla}_i \ti{\nabla}_{\ov{j}}(g^{k\ov{\ell}} \ti{g}_{k\ov{\ell}})
= -g^{i\ov{j}}\ti{\nabla}_i (g^{k\ov{q}}g^{p\ov{\ell}}\ti{g}_{k\ov{\ell}}\ti{\nabla}_{\ov{j}}g_{p\ov{q}})\\
=&g^{i\ov{j}} g^{k\ov{s}}g^{r\ov{q}}g^{p\ov{\ell}}\ti{g}_{k\ov{\ell}}\ti{\nabla}_i g_{r\ov{s}} \ti{\nabla}_{\ov{j}}g_{p\ov{q}}
+g^{i\ov{j}} g^{p\ov{s}}g^{r\ov{\ell}}g^{k\ov{q}}\ti{g}_{k\ov{\ell}}\ti{\nabla}_i g_{r\ov{s}} \ti{\nabla}_{\ov{j}}g_{p\ov{q}}\\
&-g^{i\ov{j}}g^{k\ov{q}}g^{p\ov{\ell}}\ti{g}_{k\ov{\ell}}\ti{\nabla}_i \ti{\nabla}_{\ov{j}}g_{p\ov{q}}.
\end{split}\]
But
\[\begin{split}
\ti{\nabla}_i \ti{\nabla}_{\ov{j}}g_{p\ov{q}}=&
\ti{\nabla}_i\left(\de_{\ov{j}} g_{p\ov{q}}-\ov{\ti{\Gamma}^s_{jq}}g_{p\ov{s}}\right)=
\de_i\de_{\ov{j}} g_{p\ov{q}}-\ti{\Gamma}^{r}_{ip}\de_{\ov{j}}g_{r\ov{q}}-g_{p\ov{s}}\de_i \ov{\ti{\Gamma}^s_{jq}}\\
&-\ov{\ti{\Gamma}^s_{jq}}\de_i g_{p\ov{s}}+\ti{\Gamma}^{r}_{ip}\ov{\ti{\Gamma}^s_{jq}}g_{r\ov{s}}
=\ti{R}_{i\ov{j}r\ov{q}}\ti{g}^{r\ov{s}}g_{p\ov{s}}-R_{i\ov{j}p\ov{q}}+g^{r\ov{s}} \ti{\nabla}_i g_{p\ov{s}}
\ti{\nabla}_{\ov{j}} g_{r\ov{q}},
\end{split}\]
and so
\[\begin{split}
\Delta \tr{\omega}{\ti{\omega}}&=g^{i\ov{j}} g^{k\ov{s}}g^{r\ov{q}}g^{p\ov{\ell}}\ti{g}_{k\ov{\ell}}\ti{\nabla}_i g_{r\ov{s}} \ti{\nabla}_{\ov{j}}g_{p\ov{q}}+g^{i\ov{j}}g^{k\ov{q}}g^{p\ov{\ell}}\ti{g}_{k\ov{\ell}} R_{i\ov{j}p\ov{q}}
-g^{i\ov{j}}g^{p\ov{q}}\ti{R}_{i\ov{j}p\ov{q}}.
\end{split}\]
On the other hand
\begin{align*}
\frac{\de}{\de t} \tr{\omega}{\ti{\omega}}
={} & \tr{\omega}{\ti{\omega}}+\tr{\omega}{\left(\omega_S - \tilde{\omega} \right)}+
g^{i\ov{j}}h^{p\ov{q}}R_{p\ov{q}i\ov{j}}.
\end{align*}
Therefore,
\[\begin{split}
\left( \ddt{} - \Delta \right) \tr{\omega}{\ti{\omega}}
=& \tr{\omega}{\omega_S}
+g^{i\ov{j}}g^{k\ov{\ell}}\ti{R}_{i\ov{j}k\ov{\ell}}\\
&+g^{i\ov{j}}h^{p\ov{q}}(R_{p\ov{q}i\ov{j}}-R_{i\ov{j}p\ov{q}})
-g^{i\ov{j}}g^{k\ov{\ell}}h^{p\ov{q}}
\ti{\nabla}_i g_{k\ov{q}} \ti{\nabla}_{\ov{j}} g_{p\ov{\ell}}.
\end{split}\]

From Lemma \ref{lemmatildegestimates},  we conclude that
\begin{equation}\label{uno}
\left( \ddt{} - \Delta \right) \tr{\omega}{\tilde{\omega}}\le Ce^{t/2}
+g^{i\ov{j}}h^{p\ov{q}}(R_{p\ov{q}i\ov{j}}-R_{i\ov{j}p\ov{q}})-g^{i\ov{j}}g^{k\ov{\ell}}h^{p\ov{q}}
\ti{\nabla}_i g_{k\ov{q}} \ti{\nabla}_{\ov{j}} g_{p\ov{\ell}}.
\end{equation}
We also use the equivalence of $g$ and $h$ to bound
\begin{equation}\label{due}
-g^{i\ov{j}}g^{k\ov{\ell}}h^{p\ov{q}}
\ti{\nabla}_i g_{k\ov{q}} \ti{\nabla}_{\ov{j}} g_{p\ov{\ell}}\leq -C_0^{-1}|\ti{\nabla} g|^2_g.
\end{equation}
Next we have
\begin{equation}\label{sette}
|g^{i\ov{j}}h^{p\ov{q}}(R_{p\ov{q}i\ov{j}}-R_{i\ov{j}p\ov{q}})|\leq Ce^{t/2}+\frac{1}{2C_0}|\ti{\nabla} g|^2_g.
\end{equation}
Indeed, this follows from the same argument as in the proof of \eqref{sette0}, even though the tensor $h$ there is different.
Combining \eqref{uno}, \eqref{due}, \eqref{sette}, we get the desired inequality.
\end{proof}

Next we use the exponential decay of $\varphi$ (Lemma \ref{key}) and $\dot{\varphi}$ (Lemma \ref{expdotphi}) to obtain an exponential decay bound from
above for $\tr{\omega}{\tilde{\omega}}-2$
and $\tr{\ti{\omega}}{\omega}-2$.

\begin{proposition} \label{proppinch} For any $0<\ve<1/4$ there is a constant $C>0$ such that for $t \ge T_I$,
\begin{equation} \label{firsttr}
\tr{\omega}{\tilde{\omega}} - 2  \le C e^{-\ve t}
\end{equation}
and
\begin{equation} \label{secondtr}
\emph{tr}_{\tilde{\omega}}{\omega} -2 \le Ce^{-\ve t}.
\end{equation}
\end{proposition}

\begin{proof}
Given $0<\ve<1/4$, choose $\eta>\ve$ such that $\ve+1/2+\eta<1$ (in particular $0<\eta<1/2$), and choose $\delta$ satisfying
$2\ve+1/2<\delta<\ve+1/2+\eta$ (in particular $0<\delta<1$), which we can do because $\ve<\eta$.
For \eqref{firsttr}, we compute the evolution of
$$Q_1 = e^{\ve t} (\textrm{tr}_{\omega}{\tilde{\omega}} - 2) -  e^{\delta t} \varphi.$$
From Lemma \ref{key}, it suffices to obtain a uniform upper bound for $Q_1$.
Compute using Lemma \ref{schwarztype} and the fact that $\Delta \varphi = 2 - \tr{\omega}{\tilde{\omega}}$,
\begin{equation}\label{evolveQ}\begin{split}
\left( \ddt{} - \Delta \right) Q _1\le & C e^{(\ve+1/2)t} + \ve e^{\ve t} (\tr{\omega}{\tilde{\omega}} - 2) - \delta e^{\delta t} \varphi  - e^{\delta t} \dot{\varphi} \\
& -  e^{\delta t} (\tr{\omega}{\tilde{\omega}} -2)\\
 \le  & C e^{(\ve+1/2)t}   +  Ce^{(\delta - \eta) t}  -  e^{\delta t} (\tr{\omega}{\tilde{\omega}} -2),
\end{split}\end{equation}
where in the last line we have used the lower bound $\dot{\varphi} \ge -Ce^{-\eta t}$  from  Lemma \ref{expdotphi} (since $0<\eta<1/2$),  Lemma \ref{key} and the fact that $\tr{\omega}{\tilde{\omega}} \le C$.

But we have $\delta - \eta<\ve + 1/2$ and so
 at a maximum point of $Q_1$,
$$0 \le C e^{(\ve+1/2)t}-  e^{\delta t} (\tr{\omega}{\tilde{\omega}} -2),$$
and hence
$$e^{\ve t} (\tr{\omega}{\tilde{\omega}} -2) \le C e^{(2\ve - \delta + 1/2)t} \le C.$$
since we chose $\delta$ so that $2\ve - \delta +1/2 <0$.  This implies that $Q_1$ is bounded from above at any maximum point, and completes the proof of \eqref{firsttr}.

The proof  of \eqref{secondtr} is slightly more complicated.  First recall that (see \eqref{nice}) for $t \ge T_I$,
$$\left( \ddt - \Delta \right) \tr{\tilde{\omega}}{\omega} \le C e^{t/2}.$$
Fix $\sigma$ with $0< \ve < \sigma < 1/4$.  From Lemma \ref{expdotphi} we have
\begin{equation} \label{phidb}
-C e^{-\sigma t} \le \dot{\varphi} \le C e^{-\sigma t}.
\end{equation}

Now choose $\delta$ with $1/2 + 2 \ve < \delta < 1/2 + \ve + \sigma$, which we can do because $\ve<\sigma$.  Then set
$$Q_2 = e^{\ve t} (\tr{\tilde{\omega}}{\omega} -2) - e^{\delta t}\varphi.$$
Compute
\begin{equation} \label{eqnub1}
\begin{split}
\left( \ddt{} - \Delta \right)Q_2 \le  & C e^{(\ve+1/2)t} - e^{\delta t} \dot{\varphi} - e^{\delta t} (\tr{\omega}{\tilde{\omega}}-2) \\
\le & Ce^{(\ve+1/2)t} - e^{\delta t} (\tr{\omega}{\tilde{\omega}} -2),
\end{split}
\end{equation}
using \eqref{phidb} and the fact that $\delta-\sigma <1/2+\ve$.
We now wish to replace the term $\tr{\omega}{\tilde{\omega}}$ by the sum of $\tr{\tilde{\omega}}{\omega}$ and a small error term.
\begin{equation} \label{eqnub2}
\tr{\omega}{\tilde{\omega}} = \frac{ \tilde{\omega}^2}{\omega^2} \tr{\tilde{\omega}}{\omega} = \tr{\tilde{\omega}}{\omega} + \left(  \frac{\tilde{\omega}^2}{\omega^2} -1 \right) \tr{\tilde{\omega}}{\omega}.
\end{equation}
Then from \eqref{CRF}, \eqref{claimOmega} and Lemma \ref{key},
$$\dot{\varphi} = \log  \frac{\omega^2}{\tilde{\omega}^2} + O(e^{-\sigma t}),$$
and so using \eqref{phidb} again
$$\left| \frac{\tilde{\omega}^2}{\omega^2} -1 \right| = \left| e^{ O(e^{-\sigma t})}  -1 \right| \le C e^{-\sigma t},$$
which implies that, since $\tr{\tilde{\omega}}{\omega}$ is uniformly bounded,
\begin{equation} \label{eqnub3}
\left| \left(  \frac{\tilde{\omega}^2}{\omega^2} -1 \right) \tr{\tilde{\omega}}{\omega} \right| \le Ce^{-\sigma t}.
\end{equation}
Then combining  \eqref{eqnub1}, \eqref{eqnub2}, \eqref{eqnub3} and again using the fact that $e^{(\delta - \sigma)t} \le e^{(\ve+1/2)t}$ we obtain for $t\ge T_I$,
\begin{equation*}
 \begin{split}
\left( \ddt{} - \Delta \right)Q_2 \le  & Ce^{(\ve+1/2)t} - e^{\delta t} (\tr{\tilde{\omega}}{\omega}-2) - e^{\delta t} \left( \frac{\tilde{\omega}^2}{\omega^2} -1 \right) \tr{\tilde{\omega}}{\omega} \\
\le  & C e^{(\ve+1/2)t}  - e^{\delta t} (\tr{\tilde{\omega}}{\omega}-2).
\end{split}
\end{equation*}
Then at the maximum point of $Q_2$ (occurring at a time $t>T_I$) we have
$$e^{\ve t} (\tr{\tilde{\omega}}{\omega}-2)  \le C e^{(2\ve + 1/2 - \delta) t} \le C',$$
since we chose $\delta>1/2+2\ve$.  This shows that $Q_2$ is bounded from above and completes the proof.
\end{proof}

To show that $\omega$ and $\tilde{\omega}$ approach each other exponentially fast we use an elementary lemma:

\begin{lemma} \label{lemmael}
Let $\ve>0$ be small.
Suppose that
$$\emph{tr}_{\omega}{\tilde{\omega}} - 2 \le \ve \quad \textrm{and} \quad \emph{tr}_{\tilde{\omega}}{\omega} - 2 \le \ve.$$  Then
$$(1-2\sqrt{\ve}) \tilde{\omega} \le \omega \le (1+2\sqrt{\ve}) \tilde{\omega}.$$
\end{lemma}
\begin{proof} We may work at a point at which $\tilde{g}$ is the identity and $g$ is diagonal with eigenvalues $\lambda_1, \lambda_2$.
 Then the lemma amounts to proving that if $\lambda_1, \lambda_2>0$ satisfy
$$\lambda_1 +\lambda_2 \le 2+\ve, \quad \frac{1}{\lambda_1} + \frac{1}{\lambda_2} \le 2+\ve,$$
then
$$1-2\sqrt{\ve} \le \lambda_i \le 1+ 2\sqrt{\ve}, \quad \textrm{for $i=1,2$}. $$
By symmetry, we only have to prove the estimate for $\lambda_1$.
We have
\begin{equation} \label{ll1}
\lambda_1 \le 2+\ve - \lambda_2, \quad \frac{1}{\lambda_2} \le\frac{(2+\ve) \lambda_1 - 1}{\lambda_1},
\end{equation}
which implies in particular that $(2+\ve) \lambda_1 - 1>0$.
This last inequality implies that
$$- \lambda_2 \le \frac{-\lambda_1}{(2+\ve)\lambda_1-1}.$$
Then in \eqref{ll1},
$$\lambda_1 \le 2+ \ve - \frac{\lambda_1}{(2+\ve)\lambda_1-1}.$$
Multiplying this by $(2+\ve) \lambda_1 - 1>0$ and simplifying, we get
$$\lambda_1^2 - (2+\ve)\lambda_1 + 1 \le 0.$$
Completing the square, we obtain
$$(\lambda_1 - (1+ \ve/2))^2 \le \ve + \ve^2/4.$$
Then, assuming $\ve>0$ is  smaller than some universal constant,
$$1-2\sqrt{\ve} \le \lambda_1 \le 1+2\sqrt{\ve},$$
as required.
\end{proof}

Finally, we complete the proof of Theorem \ref{theorempinch}:

\begin{proof}[Proof of Theorem \ref{theorempinch}]
Combine Proposition \ref{proppinch} with Lemma \ref{lemmael}.\end{proof}

\section{A third order estimate} \label{sectioncalabi}

In this section we prove a ``Calabi-type'' estimate for the first derivative of the evolving metric.
One might guess that the natural quantity to consider is $| \tilde{\nabla} g|^2_g$, following the computation in \cite{SW}, say.  However, we encountered difficulties in obtaining a good bound for this quantity because of the non-K\"ahlerity of the reference metrics $\tilde{\omega}$.  Our idea then is to take a \emph{K\"ahler} reference metric.  Of course, in general $M$ may not admit a global K\"ahler metric, so we work locally on an open set where the bundle $M$ is trivial.

We obtain a Calabi estimate on this open set, using a cut-off function and a local reference K\"ahler metric.  Our computations are based on those in  \cite{ShW}.  However, the situation here is complicated by the fact that the metrics are collapsing in the fiber directions, and we will need to make careful use of the bounds from Lemma \ref{lemmatildegestimates}.

Fix a point $y\in S$ and neighborhood $B$ of $y$ over which $\pi$ is trivial, so $U=\pi^{-1}(B)\cong B\times E$.
Over $U$ we have $\omega_E=i\alpha\wedge\ov{\alpha}$ a
$d$-closed semi-flat $(1,1)$-form constructed as in the proof of Lemma \ref{good}. Therefore, $\hat{\omega}=\omega_E+\omega_S$ is a
semi-flat product K\"ahler metric on $U$. From now on we work exclusively on $U$, where we define $\mathcal{S}=|\hat{\nabla} g|^2_g$.  Fix a smaller open set $V \subset\subset U$.

\begin{theorem}\label{c3}
On $V$ we have
$$\mathcal{S}\leq Ce^{2t/3},$$
for all $t>0$.
\end{theorem}

By compactness, we obtain the same bound in any such neighborhood $V$.
Recall that  $\omega_{\mathrm{flat},y}$ denotes the unique flat metric on the fiber $E_y=\pi^{-1}(y)$ in the K\"ahler class $[\omega_0|_{E_y}]$.
Exactly as in \cite[Lemma 6.9]{SW} (see also \cite{G2}, \cite[Theorem 1.1]{GTZ}, \cite[Proposition 5.8]{FZ}) we have that:

\begin{corollary}\label{fiberconv}
For any $y\in S$, we have on $E_y$,
$$e^t\omega(t)|_{E_y}\to \omega_{\mathrm{flat},y}$$
exponentially fast in the $C^1(E_y,g_0)$ topology. Moreover, the convergence is uniform in $y\in S$.
\end{corollary}
\begin{proof}
We use an idea from \cite[p. 440]{To}. Since $g|_{E_y}$ is uniformly equivalent to $e^{-t}\hat{g}|_{E_y}=e^{-t}g_E$,
we conclude that
\[\begin{split}
|\nabla_{g_E} (e^t g|_{E_y})|^2_{g_E}&=e^{-t}|\nabla_{g_E} (g|_{E_y})|^2_{e^{-t}g_E}
\leq Ce^{-t} |\nabla_{\hat{g}|_{E_y}} (g|_{E_y})|^2_{g}\\
&\leq Ce^{-t}\mathcal{S}\leq Ce^{-t/3},
\end{split}\]
using Theorem \ref{c3}. But on $E_y$, $g_{\mathrm{flat},y}$ is a constant multiple of $g_E$, and so
$$|\nabla_{g_E} (e^t g|_{E_y}-g_{\mathrm{flat},y})|^2_{g_E}\leq Ce^{-t/3}.$$
The rest of the proof follows easily, and exactly as in \cite[Lemma 6.9]{SW}, since $e^tg|_{E_y}$ and $g_{\textrm{flat},y}$ lie in the same K\"ahler class on $E_y$.
\end{proof}

Before we start the proof of Theorem \ref{c3}, we need some preliminary calculations.
Denote by $\Psi^i_{jk}=\Gamma^i_{jk}-\hat{\Gamma}^i_{jk}$, the difference of the Christoffel symbols of $g$ and $\hat{g}$.
It is a tensor which satisfies $|\Psi|^2_g=\mathcal{S}$.
The evolution of $\mathcal{S}$ is computed in \cite[(3.4)]{ShW}, generalizing calculations of \cite{Y1, Ca, PSS}, which gives
\begin{equation}\label{long}
\begin{split}
	\left(\ddt - \Delta\right) \mathcal{S} = & \mathcal{S}- |\ov{\nabla} \Psi|^2_g - |\nabla \Psi|^2_g \\
		&+ g^{i\ov{j}}g^{r\ov{s}}g^{a\ov{b}}\left( \nabla_r \ov{T_{bj \ov{a}}} + \nabla_{\ov{b}} T_{ar\ov{j}}
			\right) \Psi^k_{i p} \ov{\Psi^\ell_{s q}}g^{p\ov{q}}g_{k\ov{\ell}} \\
		&+ g^{i\ov{j}}g^{r\ov{s}}g^{a\ov{b}}\left( \nabla_r \ov{T_{bj\ov{a}}}
			+ \nabla_{\ov{b}} T_{ar\ov{j}}
			\right) \Psi_{pi}^k \ov{\Psi_{qs}^\ell}	g^{p\ov{q}}g_{k\ov{\ell}} \\
		& - g^{a\ov{b}}\left( \nabla_k \ov{T_{bs\ov{a}}} + \nabla_{\ov{b}} T_{ak\ov{s}}
			\right) \Psi_{ip}^k \ov{\Psi_{jq}^s} g^{i\ov{j}}g^{p\ov{q}}\\
		&- 2 \mathrm{Re} \bigg[ g^{r\ov{s}}\big(
			\nabla_i \nabla_p \ov{T_{s \ell\ov{r}}}
			+ \nabla_i \nabla_{\ov{s}} T_{rp\ov{\ell}}\\
		&	- T_{ir}^a R_{a\ov{s}p\ov{\ell}}
			+  g_{k\ov{\ell}}\nabla_r \hat{R}_{i\ov{s} p}{}^k
			\big) \ov{\Psi_{jq}^\ell} g^{i\ov{j}} g^{p\ov{q}}\bigg],
\end{split}\end{equation}
where $\hat{R}_{i\ov{s} p}{}^k$ is the curvature tensor of $\hat{g}$.
We are going to bound each of these terms separately. The key difference from the calculation in \cite{ShW} is that in our case the torsion $T_{ij\ov{\ell}}$
of $g$ does not equal the torsion of $\hat{g}$ (which here is zero), but rather the torsion $\ti{T}_{ij\ov{\ell}}$ of $\ti{g}$.

Therefore we let $\ti{\Psi}^i_{jk}=\Gamma^i_{jk}-\ti{\Gamma}^i_{jk}$ and $H^i_{jk}=\ti{\Gamma}^i_{jk}-\hat{\Gamma}^i_{jk}=\Psi^i_{jk}-\ti{\Psi}^i_{jk}$.

\begin{lemma}\label{bounds} For all $t\geq T_I$, we have that
\begin{equation}\label{use1}
|H|_g\leq Ce^{t/2}.
\end{equation}
\begin{equation}\label{use2}
|\db H|_g\leq Ce^{t/2}.
\end{equation}
\begin{equation}\label{use3}
|\ti{\nabla}H|_g\leq Ce^{t}.
\end{equation}
\end{lemma}
\begin{proof}
At any given point $x\in M$ we can choose local bundle coordinates $(z^1, z^2)$ as before. Then in this coordinate system \eqref{Gammabd} gives
$$|\ti{\Gamma}^i_{jk}|^2_g\leq Ce^t,$$
where we are using the fact that $g$ and $\tilde{g}$ are uniformly equivalent.
Since $\hat{g}$ is a
semi-flat product K\"ahler metric, $\hat{\Gamma}^i_{jk}$ is zero except when $i=j=k=2$,
and so
$$|\hat{\Gamma}^i_{jk}|^2_g\leq C.$$
The bound \eqref{use1} follows immediately from these bounds.

Next, note that
$$\de_{\ov{\ell}}H^i_{jk}=-\ti{R}_{j\ov{\ell}k}{}^i+\hat{R}_{j\ov{\ell}k}{}^i.$$
From Lemma \ref{lemmatildegestimates}, part (ii), we have that
$|\ti{R}_{j\ov{\ell}k}{}^i|_g\leq Ce^{t/2}$, while the fact that $\hat{g}$ is a
semi-flat product K\"ahler metric implies that the only nonzero component of $\hat{R}$ is
$\hat{R}_{2\ov{2}2}{}^2$, and so
\begin{equation}\label{curvbd}
|\hat{R}_{j\ov{\ell}k}{}^i|_g\leq C,
\end{equation}
from which \eqref{use2} follows.

We have that
$$\ti{\nabla}_p H^i_{jk}=\de_p H^i_{jk}-\ti{\Gamma}_{pj}^\ell H^i_{\ell k}-\ti{\Gamma}_{pk}^\ell H^i_{j\ell}+\ti{\Gamma}_{p\ell}^i H^\ell_{jk}.$$
Thanks to \eqref{Gammabd} and \eqref{use1}, in these coordinates we can bound the $|\cdot |_{g}$ norm of the last three terms by $Ce^t$.
As for the first term, we have
$$\de_p H^i_{jk}=\de_p \ti{\Gamma}^i_{jk}-\de_p \hat{\Gamma}^i_{jk},$$
and
$$\de_p \hat{\Gamma}^i_{jk}=\hat{g}^{i\ov{\ell}}\de_p\de_j\hat{g}_{k\ov{\ell}}-\hat{g}^{i\ov{s}}\hat{g}^{r\ov{\ell}}\de_p\hat{g}_{r\ov{s}}
\de_j\hat{g}_{k\ov{\ell}},$$
which is zero except when $i=j=k=p=2$, and so
$$|\de_p \hat{\Gamma}^i_{jk}|_g\leq C.$$
On the other hand
$$|\de_p \ti{\Gamma}^i_{jk}|_g \leq Ce^{t},$$
thanks to \eqref{claimGamma}. This proves \eqref{use3}.
\end{proof}

We now start bounding the terms in \eqref{long}.
We have
$$\nabla_{\ov{\ell}} T_{ik\ov{j}}=\ti{\nabla}_{\ov{\ell}}\ti{T}_{ik\ov{j}}-\ov{\Psi_{\ell j}^q}\ti{T}_{ik\ov{q}}+\ov{H_{\ell j}^q}\ti{T}_{ik\ov{q}},$$
and so thanks to Lemma \ref{lemmatildegestimates} and Lemma \ref{bounds} we have that
$$|\nabla_{\ov{\ell}} T_{ik\ov{j}}|_g\leq Ce^{t/2}+C\mathcal{S}^{1/2}.$$
Therefore we can bound
\begin{equation}\label{unoa}
\begin{split}
		&g^{i\ov{j}}g^{r\ov{s}}g^{a\ov{b}}\left( \nabla_r \ov{T_{bj \ov{a}}} + \nabla_{\ov{b}} T_{ar\ov{j}}
			\right) \Psi^k_{i p} \ov{\Psi^\ell_{s q}}g^{p\ov{q}}g_{k\ov{\ell}} \\
		&+ g^{i\ov{j}}g^{r\ov{s}}g^{a\ov{b}}\left( \nabla_r \ov{T_{bj\ov{a}}}
			+ \nabla_{\ov{b}} T_{ar\ov{j}}
			\right) \Psi_{pi}^k \ov{\Psi_{qs}^\ell}	g^{p\ov{q}}g_{k\ov{\ell}} \\
		& - g^{a\ov{b}}\left( \nabla_k \ov{T_{bs\ov{a}}} + \nabla_{\ov{b}} T_{ak\ov{s}}
			\right) \Psi_{ip}^k \ov{\Psi_{jq}^s} g^{i\ov{j}}g^{p\ov{q}}\leq C(e^{t/2}+\mathcal{S}^{1/2})\mathcal{S}.
\end{split}\end{equation}
Next, we compute
\[\begin{split}
	\nabla_a \nabla_b \ov{T_{ij\ov{k}}}
		=&   \nabla_a (\ti{\nabla}_b \ov{\ti{T}_{ij\ov{k}}} - \Psi_{bk}^r \ov{\ti{T}_{ij\ov{r}}} + H_{bk}^r \ov{\ti{T}_{ij\ov{r}}})  \\
		 =&  \ti{\nabla}_a \ti{\nabla}_b \ov{\ti{T}_{ij\ov{k}}}
		- \Psi_{a b}^r \ti{\nabla}_r \ov{\ti{T}_{ij\ov{k}}}
		- \Psi_{a k}^r \ti{\nabla}_b \ov{\ti{T}_{ij\ov{r}}} \\
	&+ H_{a b}^r \ti{\nabla}_r \ov{\ti{T}_{ij\ov{k}}}
		+ H_{a k}^r \ti{\nabla}_b \ov{\ti{T}_{ij\ov{r}}}  - (\nabla_a \Psi_{bk}^r) \ov{\ti{T}_{ij \ov{r}}}\\
	&	- \Psi_{bk}^r  \ti\nabla_a \ov{\ti{T}_{ij\ov{r}}}  + \Psi_{b k}^{r} \Psi_{ar}^s \ov{\ti{T}_{ij\ov{s}}} -\Psi_{b k}^{r} H_{ar}^s \ov{\ti{T}_{ij\ov{s}}}\\
&+(\ti{\nabla}_a H^r_{bk})\ov{\ti{T}_{ij\ov{r}}} + H_{bk}^r \ti{\nabla}_a\ov{\ti{T}_{ij\ov{r}}}-\Psi_{ab}^s   H_{sk}^r \ov{\ti{T}_{ij\ov{r}}} \\
&-\Psi_{ak}^s H_{bs}^r \ov{\ti{T}_{ij\ov{r}}} + H_{ab}^s   H_{sk}^r \ov{\ti{T}_{ij\ov{r}}}
+ H_{ak}^s H_{bs}^r  \ov{\ti{T}_{ij\ov{r}}}.
\end{split}\]
Using again Lemma \ref{lemmatildegestimates} and Lemma \ref{bounds} we can bound
$$|\nabla_a \nabla_b \ov{T_{ij\ov{k}}}|_g\leq C(e^t+e^{t/2}\mathcal{S}^{1/2}+\mathcal{S}+|\nabla\Psi|_g),$$
and so
\begin{equation}\label{duea}
- 2 \mathrm{Re} \bigg( g^{r\ov{s}}
			(\nabla_i \nabla_p \ov{T_{s \ell\ov{r}}} ) \ov{\Psi_{jq}^\ell} g^{i\ov{j}} g^{p\ov{q}}\bigg)\leq
C(e^t\mathcal{S}^{1/2}+e^{t/2}\mathcal{S}+\mathcal{S}^{3/2}+|\nabla\Psi|_g\mathcal{S}^{1/2}).
\end{equation}
Similarly we have
\[\begin{split}
	\nabla_a \nabla_{\ov{b}} T_{ij\ov{k}}
		=&   \nabla_a (\ti{\nabla}_{\ov{b}} \ti{T}_{ij\ov{k}} - \ov{\Psi_{bk}^r} \ti{T}_{ij\ov{r}} + \ov{H_{bk}^r} \ti{T}_{ij\ov{r}})  \\
		 =&  \ti{\nabla}_a \ti{\nabla}_{\ov{b}} \ti{T}_{ij\ov{k}}
		- \Psi_{a i}^r \ti{\nabla}_{\ov{b}} \ti{T}_{rj\ov{k}}
		- \Psi_{a j}^r \ti{\nabla}_{\ov{b}} \ti{T}_{ir\ov{k}} \\
	&+ H_{a i}^r \ti{\nabla}_{\ov{b}} \ti{T}_{rj\ov{k}}
		+ H_{a j}^r \ti{\nabla}_{\ov{b}} \ti{T}_{ir\ov{k}} - (\nabla_a \ov{\Psi_{bk}^r}) \ti{T}_{ij\ov{r}}\\
	&	- \ov{\Psi_{bk}^r}  \ti\nabla_a \ti{T}_{ij\ov{r}}  + \ov{\Psi_{b k}^{r}} \Psi_{ai}^s \ti{T}_{sj\ov{r}} + \ov{\Psi_{b k}^{r}} \Psi_{aj}^s \ti{T}_{is\ov{r}}\\
&- \ov{\Psi_{b k}^{r}} H_{ai}^s \ti{T}_{sj\ov{r}} - \ov{\Psi_{b k}^{r}} H_{aj}^s \ti{T}_{is\ov{r}}+(\ti{\nabla}_a \ov{H^r_{bk}})\ti{T}_{ij\ov{r}}\\
 &+ \ov{H_{bk}^r} \ti{\nabla}_a\ti{T}_{ij\ov{r}}-\Psi_{ai}^s  \ov{H_{bk}^r} \ti{T}_{sj\ov{r}} -\Psi_{aj}^s  \ov{H_{bk}^r} \ti{T}_{is\ov{r}} \\
 &+H_{ai}^s  \ov{H_{bk}^r} \ti{T}_{sj\ov{r}} +H_{aj}^s  \ov{H_{bk}^r} \ti{T}_{is\ov{r}}.
\end{split}\]
Using again Lemma \ref{lemmatildegestimates} and Lemma \ref{bounds} we can bound
$$|\nabla_a \nabla_{\ov{b}} T_{ij\ov{k}}|_g\leq C(e^t+e^{t/2}\mathcal{S}^{1/2}+\mathcal{S}+|\ov{\nabla}\Psi|_g),$$
and so
\begin{equation}\label{trea}
- 2 \mathrm{Re} \bigg( g^{r\ov{s}}
			\nabla_i \nabla_{\ov{s}} T_{rp\ov{\ell}} \ov{\Psi_{jq}^\ell} g^{i\ov{j}} g^{p\ov{q}}\bigg)\leq
C(e^t\mathcal{S}^{1/2}+e^{t/2}\mathcal{S}+\mathcal{S}^{3/2}+|\ov{\nabla}\Psi|_g\mathcal{S}^{1/2}).
\end{equation}
Next, we have
$$\de_{\ov{p}}\Psi^\ell_{qk}=\hat{R}_{q\ov{p}k}{}^{\ell}-R_{q\ov{p}k}{}^{\ell},$$
and so we can bound
\begin{equation}\label{quattroa}
2 \mathrm{Re} \bigg( g^{r\ov{s}} T_{ir}^a R_{a\ov{s}p\ov{\ell}} \ov{\Psi_{jq}^\ell} g^{i\ov{j}} g^{p\ov{q}}\bigg)\leq
C|\ov{\nabla}\Psi|_g \mathcal{S}^{1/2}+C|\widehat{\textrm{Rm}}|_{g}\mathcal{S}^{1/2}\leq C(1+|\ov{\nabla}\Psi|_g)\mathcal{S}^{1/2},
\end{equation}
because of \eqref{curvbd}.

Finally, we have
$$\nabla_r \hat{R}_{i\ov{s} p}{}^k=\hat{\nabla}_{r}\hat{R}_{i\ov{s} p}{}^k-\Psi_{ri}^\ell  \hat{R}_{\ell\ov{s} p}{}^k-\Psi_{rp}^\ell  \hat{R}_{i\ov{s} \ell}{}^k+\Psi_{r\ell}^k  \hat{R}_{i \ov{s} p}{}^\ell.$$
But $\hat{g}$ is a product of K\"ahler-Einstein metrics on Riemann surfaces, therefore $\hat{\nabla}_{r}\hat{R}_{i\ov{s} p}{}^k=0$. Using \eqref{curvbd} again, we conclude that
\begin{equation}\label{cinquea}
-2 \mathrm{Re} \bigg( g^{r\ov{s}}  g_{k\ov{\ell}}\nabla_r \hat{R}_{i\ov{s} p}{}^k \ov{\Psi_{jq}^\ell} g^{i\ov{j}} g^{p\ov{q}}\bigg)\leq C\mathcal{S}.
\end{equation}

Putting together \eqref{long}, \eqref{unoa}, \eqref{duea}, \eqref{trea}, \eqref{quattroa} and \eqref{cinquea}, we conclude that
\begin{equation}\label{evols}
\left(\ddt - \Delta\right) \mathcal{S} \leq C(e^{t/2}\mathcal{S}+e^t\mathcal{S}^{1/2}+\mathcal{S}^{3/2})-\frac{1}{2} |\ov{\nabla} \Psi|^2_g -\frac{1}{2} |\nabla \Psi|^2_g.
\end{equation}
Next, we define K\"ahler metrics on $U$ by $\hat{\omega}_t=e^{-t}\omega_E+\omega_S$. These are uniformly equivalent to $\omega$ independent of $t$ thanks to Theorem \ref{metricbdthm}.
Furthermore, the covariant derivative of $\hat{\omega}_t$ is independent of $t$ and equal to that of $\hat{\omega}$, and we will denote it by $\hat{\nabla}$ as before.
The same is true for the curvature of $\hat{\omega}_t$, which equals $\hat{R}_{i\ov{j}k}{}^p$.
We use \cite[Proposition 3.1]{TW} to compute
\[\begin{split}
\left( \ddt{} - \Delta \right) \tr{\hat{\omega}_t}{\omega}=&-g^{p\ov{j}}g^{i\ov{q}}\hat{g}^{k\ov{\ell}}_t\hat{\nabla}_k g_{i\ov{j}}\hat{\nabla}_{\ov{\ell}}g_{p\ov{q}}
-g^{i\ov{j}}\hat{g}^{k\ov{\ell}}_t\hat{g}^{p\ov{q}}_tg_{k\ov{q}}\hat{R}_{i\ov{\ell}p\ov{j}}\\
&-g^{i\ov{j}}\hat{g}^{k\ov{\ell}}_t \hat{\nabla}_i   \ov{\ti{T}_{j\ell\ov{k}}}   -
 g^{i\ov{j}}\hat{g}^{k\ov{\ell}}_t \hat{\nabla}_{\ov{\ell}}  \ti{T}_{ik\ov{j}}   \\
 & -\tr{\hat{\omega}_t}{\omega}+e^{-t}\hat{g}_t^{i\ov{\ell}}\hat{g}_t^{k\ov{j}}g_{i\ov{j}}(g_E)_{k\ov{\ell}}.
\end{split}\]
We have that
$$ \hat{\nabla}_i   \ov{\ti{T}_{j\ell\ov{k}}}= \ti{\nabla}_i   \ov{\ti{T}_{j\ell\ov{k}}}+H_{ik}^p  \ov{\ti{T}_{j\ell\ov{p}}},$$
and so we can use Lemma \ref{lemmatildegestimates} and \eqref{use1} to bound
$$|\hat{\nabla}_i   \ov{\ti{T}_{j\ell\ov{k}}}|_g\leq Ce^{t/2}.$$
Hence, making use of (\ref{curvbd}), we have
\begin{equation}\label{evoltr}
\left( \ddt{} - \Delta \right) \tr{\hat{\omega}_t}{\omega}\leq -C^{-1}\mathcal{S}+Ce^{t/2},
\end{equation}
for a uniform $C>0$.

We now give the proof of Theorem \ref{c3}, along the lines of \cite{ShW}.
\begin{proof} Let $K$ be a large constant such that
$$\frac{K}{2}\leq K-\tr{\hat{\omega}_t}{\omega}\leq K,$$
whose value will be fixed later.  Let $0\leq \rho\leq 1$ be a smooth nonnegative cutoff function supported in the open ball $B$ in $S$, which is identically $1$ in a smaller neighborhood
of $y$, and denote the pullback $\rho\circ\pi$ also by $\rho$.
Consider the quantity
$$Q=\rho^2\frac{e^{-2t/3}\mathcal{S}}{K-\tr{\hat{\omega}_t}{\omega}}+\tr{\hat{\omega}_t}{\omega} \qquad \textrm{on } \textrm{supp} (\rho) \subset U. $$
Our goal is to obtain an upper bound for $Q$, giving the bound $\mathcal{S} \le Ce^{2t/3}$ on a smaller neighborhood of $U$, which we may assume contains $V$.  We will apply the maximum principle to this function $Q$, noting that it is  equal to $\tr{\hat{\omega}_t}{\omega}$, and hence is bounded, on the boundary of $\textrm{supp}(\rho)$.

We start with the following observations. Since $\rho$ is the pullback of a function from the base $S$, we have from the estimate $C \omega \ge \omega_S$,
\begin{equation}\label{simple1}
|\nabla \rho|^2_g\leq C,\quad |\Delta \rho|_g\leq C,
\end{equation}
independent of $t$. Furthermore, we have the simple inequalities (see \cite[(3.9), (3.10)]{ShW}),
\begin{equation}\label{simple2}
|\nabla \tr{\hat{\omega}_t}{\omega}|^2_g\leq C\mathcal{S},\quad |\nabla \mathcal{S}|^2_g\leq 2\mathcal{S}(|\nabla\Psi|^2_g+|\ov{\nabla}\Psi|^2_g),
\end{equation}
where the first one also follows from the argument for \eqref{cs}.

We can now compute
\[\begin{split}
	\left(\ddt - \Delta\right)Q = &-\frac{2\rho^2 e^{-2t/3}\mathcal{S}}{3(K-\tr{\hat{\omega}_t}{\omega})}
+\frac{\rho^2 e^{-2t/3}}{K-\tr{\hat{\omega}_t}{\omega}}\left(\ddt - \Delta\right)\mathcal{S}\\
&+\left(1+ \rho^2\frac{e^{-2t/3}\mathcal{S}}{(K-\tr{\hat{\omega}_t}{\omega})^2}\right)\left(\ddt - \Delta\right)\tr{\hat{\omega}_t}{\omega}\\
&-2\rho^2 e^{-2t/3}\frac{\mathrm{Re}\langle\nabla \mathcal{S},\nabla  \tr{\hat{\omega}_t}{\omega}\rangle_g}{(K-\tr{\hat{\omega}_t}{\omega})^2}
-2\rho^2\frac{e^{-2t/3}\mathcal{S}}{(K-\tr{\hat{\omega}_t}{\omega})^3}|\nabla \tr{\hat{\omega}_t}{\omega}|^2_g\\
&-\Delta(\rho^2)\left(\frac{e^{-2t/3}\mathcal{S}}{K-\tr{\hat{\omega}_t}{\omega}} \right)
-4\rho \frac{e^{-2t/3}}{K-\tr{\hat{\omega}_t}{\omega}}\mathrm{Re}\left\langle \nabla\rho, \nabla \mathcal{S}\right\rangle_g\\
&-4\rho \frac{e^{-2t/3}\mathcal{S}}{(K-\tr{\hat{\omega}_t}{\omega})^2} \mathrm{Re}\left\langle \nabla\rho, \nabla \tr{\hat{\omega}_t}{\omega}\right\rangle_g.
\end{split}\]
Let $(x_0, t_0)$ be a point in $\textrm{supp}(\rho)$ at which  $Q$  achieves a maximum.  We may assume without loss of generality that $t_0>0$ and $x_0$ lies in the interior of $\textrm{supp}(\rho)$.  We have at $(x_0, t_0)$,
\[\begin{split}
2&\rho\nabla \rho \frac{e^{-2t/3}\mathcal{S}}{K-\tr{\hat{\omega}_t}{\omega}}
+\rho^2\frac{e^{-2t/3}}{K-\tr{\hat{\omega}_t}{\omega}}\nabla\mathcal{S}
+\rho^2 \frac{e^{-2t/3}\mathcal{S}}{(K-\tr{\hat{\omega}_t}{\omega})^2}
\nabla \tr{\hat{\omega}_t}{\omega} + \nabla \tr{\hat{\omega}_t}{\omega} =0.
\end{split}\]
Taking the inner product of this with $\nabla \tr{\hat{\omega}_t}{\omega}$, we see that, at this point,
\begin{equation}
\begin{split}
	0\leq \left(\ddt - \Delta\right)Q = &-\frac{2\rho^2 e^{-2t/3}\mathcal{S}}{3(K-\tr{\hat{\omega}_t}{\omega})}
+\frac{\rho^2 e^{-2t/3}}{K-\tr{\hat{\omega}_t}{\omega}}\left(\ddt - \Delta\right)\mathcal{S}\\
&+\left(1+ \rho^2\frac{e^{-2t/3}\mathcal{S}}{(K-\tr{\hat{\omega}_t}{\omega})^2}\right)\left(\ddt - \Delta\right)\tr{\hat{\omega}_t}{\omega}\\
&-\Delta(\rho^2)\left(\frac{e^{-2t/3}\mathcal{S}}{K-\tr{\hat{\omega}_t}{\omega}} \right)\\
&-4\rho \frac{e^{-2t/3}}{K-\tr{\hat{\omega}_t}{\omega}}\mathrm{Re}\left\langle \nabla\rho, \nabla \mathcal{S}\right\rangle_g
 + \frac{2 | \nabla \tr{\hat{\omega}_t}{\omega}|^2}{K-\tr{\hat{\omega}_t}{\omega}}.
\end{split} \label{eQ}
\end{equation}
We use \eqref{evols}, \eqref{evoltr}, \eqref{simple1}, \eqref{simple2}  to obtain at this point,
\[\begin{split}
0 \le  &  \frac{2\rho^2 e^{-2t/3}}{K} \left( C(e^{t/2} \mS + e^t \mS^{1/2} + \mS^{3/2}) - \frac{1}{4} |\ov{\nabla} \Psi|^2_g -\frac{1}{4} |\nabla \Psi|^2_g\right) \\
& + \left( 1+\frac{4 \rho^2 e^{-2t/3} \mS}{K^2} \right) \left( - \frac{\mS}{C} + Ce^{t/2} \right) \\
& + \frac{Ce^{-2t/3}\mS}{K} + \frac{\rho^2 e^{-2t/3}}{2K} \left(  |\ov{\nabla} \Psi|^2_g +  |\nabla \Psi|^2_g\right) + \frac{C\mathcal{S}}{K},
\end{split} \]
where we have used the Young inequality and (\ref{simple2}):
\begin{align*}
4\rho \frac{e^{-2t/3}}{K-\tr{\hat{\omega}_t}{\omega}} | \langle \nabla \rho, \nabla \mS \rangle_g| \le {} &   \frac{Ce^{-2t/3} \mS}{K}  +  \frac{\rho^2e^{-2t/3}}{4K} \frac{| \nabla \mS|^2}{\mS} \\
\le {} & \frac{Ce^{-2t/3}\mS }{K}   + \frac{\rho^2 e^{-2t/3}}{2K}  \left(  |\ov{\nabla} \Psi|^2_g +  |\nabla \Psi|^2_g\right).
\end{align*}
Suppose that at $(x_0, t_0)$ we have $e^{-2t/3}\mS K^{-1} \ge 1$, and hence $e^t \le \mS^{3/2} K^{-3/2}$ (otherwise, $e^{-2t/3}\mS$ is bounded and thus so is $Q$). We may also assume that $\mS$ is much larger than $K$, say $\mS \ge K^4$. Then at this point,
\begin{align*}
\frac{\mS}{C} + \frac{4 \rho^2 e^{-2t/3} \mS^2}{CK^2}  \le {} & \frac{C \rho^2 e^{-2t/3}}{K} \left( \frac{\mS^{7/4}}{K^{3/4}} + \frac{\mS^2}{K^{3/2}} + \frac{\mS^{2}}{K^2} \right) \\
& + \frac{C \mS^{3/4}}{K^{3/4}} + \frac{C\rho^2 e^{-2t/3} \mS^{7/4}}{K^{11/4}} + \frac{Ce^{-2t/3} \mS}{K} + \frac{C \mS}{K}
\end{align*}
Choosing $K$ to be much larger than $C^2$, we see that the second term on the left hand side of this inequality dominates all the terms involving $\rho^2$ on the right hand side.  This gives
$$\mS \left( \frac{1}{C} - \frac{C}{K^{3/4}\mathcal{S}^{1/4}} - \frac{Ce^{-2t/3}}{K} - \frac{C}{K} \right) \le 0,$$
a contradiction since we chose $K$ to be much larger than $C^2$.
It follows that  $Q$ is uniformly bounded from above at the point $(x_0, t_0)$.  This  completes the proof of the theorem.
\end{proof}

\section{Proofs of the main results} \label{sectionproofs}

In this section we give the proofs of Theorem \ref{main} and Corollary \ref{cor}.

\begin{proof}[Proof of Theorem \ref{main}]
The estimate proved in Theorem \ref{theorempinch} immediately implies that given any $0<\ve <1/8$ there exists a constant $C$ such that
$$\|\omega(t)-\ti{\omega}(t)\|_{C^0(M,g_0)}\leq Ce^{-\ve t}.$$
From the definition of $\ti{\omega}(t)=e^{-t}\omega_{\mathrm{flat}}+(1-e^{-t})\pi^*\omega_S$ we deduce that
$$\|\omega(t)-\pi^*\omega_S\|_{C^0(M,g_0)}\leq Ce^{-\ve t}.$$
The Gromov-Hausdorff convergence of $(M,\omega(t))$ to $(S,\omega_S)$ follows from Lemma \ref{GH} below.
Finally, given any $y\in S$, the exponential convergence of  $e^{t}\omega(t)|_{E_y}$ to
$\omega_{\mathrm{flat},y}$ in the $C^1(E_y, g_0)$ topology (uniformly in $y$), follows from Corollary \ref{fiberconv} and the compactness of $M$.
\end{proof}

We used the following elementary result, which is undoubtedly well-known (cf. \cite[Theorem 8.1]{TW2}).
\begin{lemma}\label{GH} Let $\pi:M\to S$ be a fiber bundle, where $(M,g_M)$ and $(S,g_S)$ are closed Riemannian manifolds.
If $g(t), t\geq 0,$ is a family of Riemannian metrics on $M$ with $\|g(t)-\pi^*g_S\|_{C^0(M,g_M)}\to 0$ as $t\to\infty$, then
$(M,g(t))$ converges to $(S,g_S)$ in the Gromov-Hausdorff sense as $t\to\infty$.
\end{lemma}
\begin{proof}
For any $y\in S$ we denote by $E_y=\pi^{-1}(y)$ the fiber over $y$. Fix $\ve>0$, denote by $L_t$ the length of a curve in $M$ measured with respect to $g(t)$,
and by $d_t$ the induced distance function on $M$. Similarly we have $L_S, d_S$ on $S$.
Using the standard formulation of Gromov-Hausdorff convergence (see e.g. \cite{TW2}), let $F=\pi:M\to S$ and define a map $G:S\to M$ by sending every point $y\in S$ to some chosen point in $M$ on the fiber $E_y$. The map $G$ will in general be discontinuous, and it
satisfies $F\circ G=\mathrm{Id}$, so
\begin{equation}\label{gh1}
d_S(y,F(G(y)))=0.
\end{equation}
On the other hand since $g(t)|_{E_y}$ goes to zero, we have that for any $t$ large and for any $x\in M$
\begin{equation}\label{gh2}
d_t(x, G(F(x)))\leq \ve.
\end{equation}
Next, given two points $x_1,x_2\in M$ let $\gamma:[0,L]\to S$ be a unit-speed minimizing geodesic in $S$ joining $F(x_1)$ and $F(x_2)$.
Since the bundle $\pi$ is locally trivial, we can cover the image of $\gamma$ by finitely many open sets $U_j, 1\leq j\leq N,$  such that
$\pi^{-1}(U_j)$ is diffeomorphic to $U_j\times E$ (where $E$ is the fiber of the bundle) and there is a subdivision $0=t_0<t_1<\dots<t_N=L$ of $[0,L]$ such that $\gamma([t_{j-1},t_j])\subset U_j$. Fix a point $e\in E$, and use the trivializations to define $\ti{\gamma}_j(s)=(\gamma(s),e)$, for $s\in [t_{j-1},t_j]$, which are curves in $M$ with the
property that
$$|L_t(\ti{\gamma}_j)-L_S(\gamma|_{[t_{j-1},t_j]})|\leq \ve/N,$$
as long as $t$ is sufficiently large (because $g(t)\to\pi^*g_S$).
The points $\ti{\gamma}_{j}(t_j)$ and $\ti{\gamma}_{j+1}(t_j)$ lie in the same fiber of $\pi$, so we can join them by a curve contained in
this fiber with $L_t$-length at most $\ve/2N$ (for $t$ large). We also join $x_1$ with $\ti{\gamma}_1(0)$ and
$x_2$ with $\ti{\gamma}_N(L)$ in the same fashion. Concatenating these ``vertical'' curves and the curves $\ti{\gamma}_j$, we obtain a piecewise
smooth curve $\ti{\gamma}$ in $M$ joining $x_1$ and $x_2$, with $\pi(\ti{\gamma})=\gamma$ and
$|L_t(\ti{\gamma})-d_S(F(x_1),F(x_2))|\leq 2\ve.$ Therefore,
\begin{equation}\label{gh3}
d_t(x_1,x_2)\leq L_t(\ti{\gamma})\leq d_S(F(x_1),F(x_2))+2\ve.
\end{equation}
Since $F\circ G=\mathrm{Id},$ we also have that for all $t$ large and for all $y_1, y_2\in S$,
\begin{equation}\label{gh4}
d_t(G(y_1),G(y_2))\leq d_S(y_1,y_2)+2\ve.
\end{equation}
Given now two points $x_1,x_2\in M$, let $\gamma$ be a unit-speed minimizing $g(t)$-geodesic joining them. If we denote by $L_{\pi^*g_S}(\gamma)$ the length of $\gamma$
using the degenerate metric $\pi^*g_S$, then we have for $t$ large,
\begin{equation}\label{gh5}
d_S(F(x_1),F(x_2))\leq L_S(F(\gamma))=L_{\pi^*g_S}(\gamma)\leq L_t(\gamma)+\ve =d_t(x_1,x_2)+\ve,
\end{equation}
where we used again that $g(t)\to\pi^*g_S$. Obviously this also implies that for all $t$ large and for all $y_1, y_2\in S$,
\begin{equation}\label{gh6}
d_S(y_1,y_2)\leq d_t(G(y_1),G(y_2))+\ve.
\end{equation}
Combining \eqref{gh1}, \eqref{gh2}, \eqref{gh3}, \eqref{gh4}, \eqref{gh5} and \eqref{gh6} we get the required Gromov-Hausdorff convergence.
\end{proof}

\begin{proof}[Proof of Corollary \ref{cor}]
The proof is similar to \cite[Theorem 8.2]{TW2}. From \cite[Lemmas 1, 2]{Br2} or \cite[Theorem 7.4]{Wa} we see that there is a finite unramified covering $p:M'\to M$ (with deck transformation group $\Gamma$) which is also a minimal properly elliptic surface $\pi':M'\to S'$ and $\pi'$ is an elliptic fiber bundle with $S'$ a compact Riemann surface of genus at least $2$. Furthermore, $\Gamma$ also acts on $S'$ (so that $\pi'$ is $\Gamma$-equivariant),
with finitely many fixed points whose union $Z$ is precisely the image of the multiple fibers of $\pi$, with quotient $S=S'/\Gamma$, and so that the quotient map $q:S'\to S$ satisfies
$q\circ\pi' = \pi\circ p.$

Denote by $\omega_{S'}$ the orbifold K\"ahler-Einstein metric on $S'$ with $\Ric(\omega_{S'})=-\omega_{S'}$.
From the description of $M$ and $M'$ as quotients of $H\times\mathbb{C}^*$, where $H$ is the upper half plane in $\mathbb{C}$ (see e.g. \cite{Ma}, \cite[Section 8]{TW2}), it follows that $\pi'^*\omega_{S'}$ is a smooth real $(1,1)$ form on $M'$, which also equals
$p^*\pi^*\omega_S$.
Indeed, if we let $z\in H$ be the variable in the upper half plane, $w\in\mathbb{C}^*$, and $y=\mathrm{Im}z$, then from the arguments in \cite[Section 8]{TW2} we see
that the form $\pi^*\omega_S$ on $M$ is induced from the form $\frac{1}{2y^2}\mn dz\wedge d\ov{z}$ on $H\times\mathbb{C}^*$, and the exact same formula holds on $M'$.

Given any Gauduchon metric $\omega_0$ on $M$, call $\omega(t)$ its evolution under the normalized Chern-Ricci flow on $M$, as before.
Let $\omega'_0=p^*\omega_0$, which is a $\Gamma$-invariant Gauduchon metric on $M'$.
If we call $\omega'(t)$ its evolution under the normalized Chern-Ricci flow on $M'$, then $\omega'(t)$ is also $\Gamma$-invariant, and equal to $p^*\omega(t)$.
Furthermore, $\Gamma$ also acts by isometries of the distance function $d_{S'}$ of $\omega_{S'}$, with quotient space $(S,d_S)$, the distance function of the orbifold metric $\omega_S$.

Now Theorem \ref{main} applied to the elliptic bundle $\pi':M'\to S'$ shows that
$(M',\omega'(t))$ converges to $(S', \omega_{S'})$ in the Gromov-Hausdorff topology. But exactly as in \cite[Theorem 8.2]{TW2} we see that the convergence happens also
in the $\Gamma$-equivariant Gromov-Hausdorff topology, and therefore by \cite[Theorem 2.1]{Fuk} or \cite[Lemma 1.5.4]{Ro} we conclude that $(M,\omega(t))$ converges to $(S,d_S)$ in the Gromov-Hausdorff topology.

Now we apply Theorem \ref{main} again to $M'$ to see that
$$\|\omega'(t)-\pi'^*\omega_{S'}\|_{C^0(M',p^*g_0)}\leq Ce^{-\ve t}.$$
Fix now an open set $U$ of $M$, small enough so that $p^{-1}(U)$ is a disjoint union of finitely many copies $U_j$ of $U$. Then
$p:U_j\to U$ is a biholomorphism for each $j$ and the $\Gamma$-action on $p^{-1}(U)$ permutes the $U_j$'s. Therefore for each $j$, the map $p:U_j\to U$ gives an isometry
between $\omega'(t)|_{U_j}$ and $\omega(t)|_{U}$, and also between $(\pi'^*\omega_{S'})|_{U_j}$ and $(\pi^*\omega_{S})|_{U}$.
Fixing one value of $j$, from $\|\omega'(t)-\pi'^*\omega_{S'}\|_{C^0(U_j,p^*g_0)}\leq Ce^{-\ve t},$ we conclude that
$$\|\omega(t)-\pi^*\omega_{S}\|_{C^0(U,g_0)}\leq Ce^{-\ve t}.$$
Covering $M$ by finitely many such open sets $U$ shows that $\omega(t)$ converges to $\pi^*\omega_S$ in the $C^0(M,g_0)$ topology.

Finally, fix any point $y\in S\backslash Z$, and let $V$ be a small open neighborhood of $y$ such that $\pi^{-1}(V)\cong V\times E$ and
$q^{-1}(V)$ is a disjoint union of finitely many copies $V_j$ of $V$ with points $y_j\in V_j$ mapping to $y$ and with $q:V_j\to V$ a biholomorphism. Then $\pi'^{-1}(V_1)\cong V_1\times E$,
and under these identifications the biholomorphism $p:\pi'^{-1}(V_1)\to \pi^{-1}(V)$ equals $(q,\mathrm{Id}):V_1\times E\to V\times E$.
Under this map the fiber $E'_{y_1}:=\pi'^{-1}(y_1)$ is carried to the fiber $E_y$. Applying Theorem \ref{main} to $M'$, we
see that $e^{t}\omega'(t)|_{E'_{y_1}}$ converges exponentially fast in the $C^1(E'_{y_1}, g'_0)$ topology to $\omega'_{\mathrm{flat},y_1}$, the
flat K\"ahler metric on $E'_{y_1}$ cohomologous to $[\omega'_0|_{E'_{y_1}}]$, and the convergence is uniform when varying $y_1$. But the local biholomorphism $p$ maps
$\omega'(t)$ to $\omega(t)$, and $\omega'_{\mathrm{flat},y_1}$ to $\omega_{\mathrm{flat},y}$, and the result follows.
\end{proof}

\bigskip
\noindent
{\bf Acknowledgements.}  \ The authors thank the referee for some suggestions which improved the presentation.


\begin{thebibliography}{99}
\bibitem{bhpv} Barth, W.P., Hulek, K., Peters, C.A.M., Van de Ven, A. {\em Compact complex surfaces. Second edition}, Springer, 2004.
\bibitem{Br1}  Br\^inz\u{a}nescu, V. {\em N\'eron-Severi group for nonalgebraic elliptic surfaces. I. Elliptic bundle case}, Manuscripta Math. {\bf 79} (1993), no. 2, 187--195.
\bibitem{Br2} Br\^inz\u{a}nescu, V. {\em N\'eron-Severi group for nonalgebraic elliptic surfaces. II. Non-K\"ahlerian case}, Manuscripta Math. {\bf 84} (1994), no. 3-4, 415--420.
\bibitem{Ca} Cao, H.-D. {\em Deformation of K\"ahler metrics to K\"ahler-Einstein metrics on compact K\"ahler manifolds}, Invent. Math. {\bf 81}  (1985), no. 2, 359--372.
\bibitem{CY}  Cheng, S.Y., Yau, S.-T. {\em Differential equations on Riemannian manifolds and their geometric applications}, Comm. Pure Appl. Math. {\bf 28} (1975), no. 3, 333--354.
\bibitem{Ch} Cherrier, P. {\em \'Equations de Monge-Amp\`ere sur les vari\'et\'es Hermitiennes compactes}, Bull. Sc. Math. (2) {\bf 111} (1987), 343--385.
\bibitem{FIK} Feldman, M., Ilmanen, T., Knopf, D. {\em Rotationally symmetric shrinking and expanding gradient K\"ahler-Ricci solitons},  J. Differential Geom. {\bf 65}  (2003),  no. 2, 169--209.
\bibitem{Fi} Fine, J. {\em Fibrations with constant scalar curvature K\"ahler metrics and the CM-line bundle}, Math. Res. Lett. {\bf 14} (2007), no. 2, 239--247.
\bibitem{FZ} Fong, F.T.-H., Zhang, Z. {\em The collapsing rate of the K\"ahler-Ricci flow with regular infinite time singularity}, to appear in J. Reine. Angew. Math., arXiv:1202.3199.
\bibitem{Fuk} Fukaya, K. {\em Theory of convergence for Riemannian orbifolds}, Japan. J. Math. (N.S.) {\bf 12} (1986), no. 1, 121--160.
\bibitem{G} Gill, M. {\em Convergence of the parabolic complex Monge-Amp\`ere equation on compact Hermitian manifolds}, Comm. Anal. Geom. {\bf 19} (2011), no. 2, 277--303.
\bibitem{G2} Gill, M. {\em Collapsing of products along the K\"ahler-Ricci flow}, Trans. Amer. Math. Soc. {\bf 366} (2014), no. 7, 3907--3924.
\bibitem{GTZ} Gross, M., Tosatti, V., Zhang, Y. {\em Collapsing of abelian fibred Calabi-Yau manifolds}, Duke Math. J. {\bf 162} (2013), no. 3, 517--551.
\bibitem{Ha} Hamilton, R.S. {\em Three-manifolds with positive Ricci curvature}, J. Differential Geom. {\bf 17} (1982), no. 2, 255--306.
\bibitem{In} Inoue, M. {\em On surfaces of Class $VII_0$}, Invent. Math. {\bf 24} (1974), 269--310.
\bibitem{LYZ} Li, J., Yau, S.-T., Zheng, F. {\em On projectively flat Hermitian manifolds}, Comm. Anal. Geom. {\bf 2} (1994), 103--109.
\bibitem{Ma} Maehara, K. {\em On elliptic surfaces whose first Betti numbers are odd}, in {\em Proceedings of the International Symposium on Algebraic Geometry (Kyoto Univ., Kyoto, 1977),} 565--574, Kinokuniya Book Store, Tokyo, 1978.
\bibitem{Na} Nakamura, I. {\em On surfaces of class $\rm VII\sb 0$ with curves}, Invent. Math. {\bf 78} (1984), no. 3, 393--443.
\bibitem{PSS} Phong, D.H., Sesum, N., Sturm, J. {\em Multiplier ideal sheaves and the K\"ahler-Ricci flow},
Comm. Anal. Geom. {\bf 15} (2007), no. 3, 613--632.
\bibitem{PS} Phong, D.H., Sturm, J. {\em The Dirichlet problem for degenerate complex Monge-Ampere equations}, Comm. Anal. Geom. {\bf 18} (2010), no. 1, 145--170.
\bibitem{Ro} Rong, X. {\em Convergence and collapsing theorems in Riemannian geometry}, in {\em Handbook of geometric analysis, No. 2}, 193--299, Adv. Lect. Math. (ALM), {\bf 13}, Int. Press, Somerville, MA, 2010.
\bibitem{SeT} Sesum, N., Tian, G. {\em Bounding scalar curvature and diameter along the K\"ahler Ricci flow (after Perelman)}, J. Inst. Math. Jussieu {\bf 7} (2008), no. 3, 575--587.
\bibitem{ShW} Sherman, M., Weinkove, B. {\em Local Calabi and curvature estimates for the Chern-Ricci flow}, New York J. Math. {\bf 19} (2013), 565--582.
\bibitem{ST} Song, J., Tian, G. {\em The K\"ahler-Ricci flow on surfaces of positive Kodaira dimension}, Invent. Math. {\bf 170} (2007), no. 3, 609--653.
\bibitem{ST2} Song, J., Tian, G. {\em Canonical measures and K\"ahler-Ricci flow}, J. Amer. Math. Soc. {\bf 25} (2012), no. 2, 303--353.
\bibitem{ST3} Song, J., Tian, G. {\em Bounding scalar curvature for global solutions of the K\"ahler-Ricci flow}, arXiv:1111.5681.
\bibitem{SW1} Song, J., Weinkove, B. {\em Contracting exceptional divisors by the K\"ahler-Ricci flow}, Duke Math. J. {\bf 162} (2013), no. 2, 367--415.
\bibitem{SW} Song, J., Weinkove, B. {\em An introduction to the K\"ahler-Ricci flow}, in {\em An introduction to the K\"ahler-Ricci flow}, 89--188, Lecture Notes in Math. {\bf 2086}, Springer, Cham., 2013.
\bibitem{StT} Streets, J., Tian, G. {\em A parabolic flow of pluriclosed metrics}, Int. Math. Res. Not. {\bf 2010} (2010), no. 16, 3103--3133.
\bibitem{T0} Teleman, A. {\em Projectively flat surfaces and Bogomolov's theorem on class $VII_{0}$-surfaces},  Int. J. Math. {\bf 5} (1994), 253--264.
\bibitem{T1} Teleman, A. {\em Donaldson theory on non-K\"ahlerian surfaces and class VII surfaces with $b_2=1$}, Invent. Math. {\bf 162} (2005), no. 3, 493--521.
\bibitem{TZha} Tian, G., Zhang, Z. {\em On the K\"ahler-Ricci flow on projective manifolds of general type},  Chinese Ann. Math. Ser. B  {\bf 27}  (2006),  no. 2, 179--192.
\bibitem{To0} Tosatti, V. {\em A general Schwarz lemma for almost-Hermitian manifolds}, Comm. Anal. Geom. {\bf 15} (2007), no. 5, 1063--1086.
\bibitem{To} Tosatti, V. {\em Adiabatic limits of Ricci-flat K\"ahler metrics}, J. Differential Geom. {\bf 84} (2010), no.2, 427--453.
\bibitem{TW3} Tosatti, V., Weinkove, B. {\em The complex Monge-Amp\`ere equation on compact Hermitian manifolds}, J. Amer. Math. Soc. {\bf 23} (2010), no.4, 1187--1195.
\bibitem{TW} Tosatti, V., Weinkove, B. \emph{On the evolution of a Hermitian metric by its Chern-Ricci form}, to appear in J. Differential Geom., arXiv:1201.0312.
\bibitem{TW2} Tosatti, V., Weinkove, B. {\em The Chern-Ricci flow on complex surfaces},  Compos. Math. {\bf 149} (2013), no. 12, 2101--2138.
\bibitem{Wa} Wall, C.T.C. {\em Geometric structures on compact complex analytic surfaces}, Topology {\bf 25} (1986), no. 2, 119--153.
\bibitem{Y1}Yau, S.-T. {\em On the Ricci curvature of a compact K\"ahler manifold and the complex Monge-Amp\`ere equation}, I. Comm. Pure Appl. Math. {\bf 31} (1978), no. 3, 339--411.
\bibitem{Ya}Yau, S.-T. {\em A general Schwarz lemma for K\"ahler manifolds}, Amer. J. Math. {\bf 100} (1978), no. 1, 197--203.

\end{thebibliography}
\end{document}